\documentclass[11pt]{article}
\usepackage{lmodern}
\usepackage[margin=0.8in]{geometry}

\usepackage{tocloft}

\usepackage{booktabs}\usepackage[svgnames]{xcolor}
\usepackage[bookmarksnumbered=true]{hyperref} 
\hypersetup{
     colorlinks = true,
     linkcolor = blue,
     anchorcolor = blue,
     citecolor = teal,
     filecolor = blue,
     urlcolor = blue
     }

   \usepackage{amsmath}
   \usepackage{amssymb}
   \usepackage{bm}
\usepackage{soul}
\usepackage{amsmath}
\usepackage{amssymb}
\usepackage{amsthm}
\usepackage{wasysym}
\usepackage{mathtools}
\usepackage{graphicx}
\usepackage{hyperref}
\usepackage{booktabs}
\usepackage{bbold}
\usepackage{mathrsfs}
\usepackage{tikz-cd}
\usepackage{color}
\usepackage{setspace}
\usepackage{comment}
\usepackage{multirow}
\usepackage{bm}
\usepackage{thmtools}
\usepackage{thm-restate}
\usepackage{cleveref}
\usepackage{enumitem}
\usepackage{csquotes}
\usepackage{longtable}
\usepackage{xypic}
\usepackage{mathdots}
\usepackage{rotating}

\newcommand{\NN}{\mathbb{N}}
\newcommand{\ZZ}{\mathbb{Z}}

\newcommand{\QQ}{\mathbb{Q}}
\newcommand{\RR}{\mathbb{R}}

\newcommand{\kk}{\mathbb{k}}
\newcommand{\Pinfty}{\mathcal{P}_{\infty}}
\newcommand{\mD}{\mathcal{D}}
\newcommand{\mC}{\mathcal{C}}
\newcommand{\mM}{\mathcal{M}}
\newcommand{\mS}{\mathcal{S}}
\newcommand{\mT}{\mathcal{T}}
\newcommand{\mA}{\mathcal{A}}
\newcommand{\mP}{\mathcal{P}}
\newcommand{\pwf}{\rm{pwf}}
\newcommand{\fp}{\rm{fp}}

\DeclareMathOperator{\interior}{int}
\DeclareMathOperator{\supp}{supp}
\DeclareMathOperator{\id}{id}
\DeclareMathOperator{\Rep}{Rep}
\DeclareMathOperator{\rep}{rep}
\DeclareMathOperator{\Ext}{Ext}
\DeclareMathOperator{\Hom}{Hom}
\DeclareMathOperator{\add}{add}

\DeclareMathOperator{\ind}{ind}
\DeclareMathOperator{\End}{End}
\DeclareMathOperator{\rad}{rad}

\theoremstyle{definition}
\newtheorem{definition}{Definition}[section]
\newtheorem{remark}[definition]{Remark}

\numberwithin{subcase}{case}
\newtheorem*{notation}{Notation}

\theoremstyle{plain}
\newtheorem{theorem}[definition]{Theorem}
\newtheorem{lemma}[definition]{Lemma}

\newtheorem{proposition}[definition]{Proposition}
\newtheorem{conjecture}[definition]{Conjecture}

\usepackage{newtxmath}      % provides definitions for: \jmath, \amalg, \coprod
\hypersetup{
     colorlinks = true,
     linkcolor = blue,
     anchorcolor = blue,
     citecolor = teal,
     filecolor = blue,
     urlcolor = black
     }

\usepackage{enumitem} 
\newlist{condenum}{enumerate}{1} 
\setlist[condenum]{label=\bfseries Condition \arabic*., 
                   ref=\arabic*, wide}

\title{Progress on 
Infinite Cluster Categories Related to Triangulations of the (Punctured) Disk
}
\author{Fatemeh Mohammadi, Job Daisie Rock, and Francesca Zaffalon}

\begin{document}

\maketitle

\begin{abstract}
    In this mostly expository paper, we present recent progress on infinite (weak) cluster categories that are related to triangulations of the disk, with and without a puncture.
    First we recall the notion of a cluster category.
    Then we move to the infinite setting and survey recent work on infinite cluster categories of types $\mathbb{A}$ and $\mathbb{D}$.
    We conclude with our contributions, two infinite families of infinite (weak) cluster categories of type $\mathbb{D}$.
    We first present a discrete, infinite version of Schiffler's combinatorial model of the punctured disk with marked points.
    We then produce each (weak) cluster category starting with representations of thread quivers, taking the derived category, and then taking the appropriate orbit category.
    We show that the combinatorics in the (weak) cluster categories match with the corresponding combinatorics of the punctured disk with countably-many marked points.
    We also state two conjectures concerning weak cluster structures inside our (weak) cluster categories.
\end{abstract}

%\noindent{keywords: cluster categories; infinite cluster categories; cluster structures; triangulations.}

%\ccode{MSC2020: 13F60, 16G20, 16G70}
\setcounter{tocdepth}{2}
{\hypersetup{linkcolor=black}
{\tableofcontents}}

\section{Introduction}

\noindent{\bf 
Motivation.}
Cluster algebras are a family of commutative rings introduced by Fomin and Zelevinsky \cite{fomin2002cluster}. They were originally introduced to supply an algebraic framework for the study of Lusztig's total positivity \cite{lusztig2010introduction, lusztig1994total, fomin1999double}. Since then cluster algebras and their applications in other areas have been extensively studied. In particular, they naturally arise in the theory of reductive Lie groups, Poisson geometry, moduli spaces of Riemann surfaces,
tilting theory, quantum physics, and scattering amplitudes; see e.g.~the introduction of \cite{caldero2008triangulated} and the references therein. One important family of cluster algebras are the homogeneous coordinate rings of Grassmannian varieties; see \cite{fomin2002clusterfinitetype,scott2006grassmannians}.
In \cite{galashin2019positroid}, Galashin and Lam extended Scott's results \cite{scott2006grassmannians} from Grassmannians to positroid varieties which are subvarieties of the Grassmannians; see also \cite{muller2017twist}. This furthermore connects to the Amplituhedron theory and computing the scattering amplitudes of planar $\mathcal{N} = 4$ super Yang-Mills theory \cite{arkani2018scattering, arkani2016grassmannian, golden2014motivic,  williams2021positive, lukowski2019cluster, drummond2020tropical, mohammadi2021triangulations}. 

\medskip
\noindent{\bf 
Finite type cluster algebras.}
A cluster algebra is a subring of the field of rational functions $\mathbb{F} = \QQ(u_1,\dots,u_n)$, generated by a family of distinguished generators called cluster variables. These generators are inductively constructed via a so-called mutation process from an initial fixed family of indeterminates. A cluster algebra is finite if the number of cluster variables is finite.
Cluster algebras of finite type are classified in terms of Dynkin diagrams; see \cite{fomin2002clusterfinitetype}. In particular, the only Grassmannians Gr$(k,n)$ with $2\leq k\leq n/2$, whose coordinate rings are finite type cluster algebras are Gr$(2,n)$, Gr$(3, 6)$, Gr$(3, 7)$, and Gr$(3, 8)$, where their associated cluster algebras correspond to the root systems of type $\mathbb{A}_n$, $\mathbb{D}_4$, $\mathbb{E}_6$, and $\mathbb{E}_8$, respectively; see \cite{fomin2002clusterfinitetype, scott2006grassmannians}. For instance, the cluster algebra associated with a given orientation of the Dynkin diagram $\mathbb{A}_n$ can be represented by triangulations of the $(n + 3)$-gon; see \cite{caldero2006quivers}. In particular, the mutable cluster variables correspond to the diagonals of the $(n + 3)$-gon, and any mutation corresponds to flipping a diagonal in a triangulation such that the result is also a triangulation. In this combinatorial model, every diagonal $\{i,j\}$ represents a Pl\"ucker coordinate $p_{ij}$ in Gr$(2,n)$ and flipping the diagonals $\{i,j\}$ and $\{k,\ell\}$ corresponds to the $3$-term Pl\"ucker relation: $p_{ij}p_{k\ell}=p_{ik}p_{j\ell}-p_{jk}p_{i\ell}$. Cluster algebras are also studied in the context of Gr\"obner degenerations of Grassmannians; see e.g.~\cite{gross2018canonical, bossinger2020toric, bossinger2021families}.

\medskip
\noindent{\bf 
Infinite type cluster algebras.}
The study of infinite type cluster algebras is relatively underdeveloped, despite their applications in various areas. For example, it relates to the subdivisions of the Amplituhedron, which is a geometric object defined by Arkani-Hamed and Trnka \cite{arkani2016grassmannian} whose subdivisions have profound use cases in physics, especially in computing scattering amplitudes of particles. 
In particular, positroid cells \cite{postnikov2006total}
are good candidates for decomposing the Amplituhedron, where the cluster variables in Gr$(k,n)$ describe the (facet) structures of such positroid cells; see \cite{lukowski2019cluster}.  

\medskip

\noindent{\bf Cluster categories.} The finite type cluster algebras have a combinatorial counterpart, which encodes their algebraic properties. Cluster categories are introduced to develop the dictionary between the combinatorial features of these objects and the properties of the cluster algebras; see \cite{buan2009cluster, buan2007cluster}. For example, in \cite{caldero2006quivers}, Caldero, Chapoton and Schiffler associated a cluster category to the Dynkin quivers of type $\mathbb{A}_n$. This has been further generalized by Holm and Jørgensen leading to the identification of a cluster category in \cite{holm2012cluster} of type $\mathbb{A}_\infty$. Moreover, its cluster tilting subcategories are in correspondence with triangulations of the infinity-gon; see also \cite{grabowski2014cluster}. Note that the Auslander-Reiten quiver can be defined for any abelian category, where the vertices represent
the indecomposable objects of the category and arrows the
irreducible morphisms among them. Since Holm and Jørgensen's work, cluster categories for infinite (or continuous) quivers have been extensively studied, with the hope of identifying the analogous cluster categories for other infinite Dynkin diagrams; see e.g.~\cite{igusa2015continuous, igusa2015cyclic, igusa2022continuous, berg2014threaded}. We continue this study for infinite type $\mathbb{D}$ cluster categories.
We note that there is a discrete infinite cluster category of type $\mathbb{D}$ by Yang \cite{yang2017cluster} and a continuous cluster category of type $\mathbb{D}$ by Igusa and Todorov \cite{igusa2013continuous}. See \S\ref{sec:infinite  cluster categories} for more details.

\medskip
\noindent{\bf 
Our contribution.} We construct a family of cluster categories of countable type $\mathbb{D}$ and a family of weak cluster categories of infinite type $\mathbb{D}$.
In particular, our main results below shows that, algebraically, there is a good justification for extending the notion of punctured $n$-gon from \cite{schiffler2008geometric} to the infinite setting. 

\medskip
Before stating our main results, we generalize the notion of punctured disk by introducing two families of combinatorial models $\mathcal{P}_{n,\infty}$ and $\mathcal{P}_{\overline{n,\infty}}$, where:
\begin{itemize}
\item $\mathcal{P}_{n,\infty}$ denotes the punctured disk with infinite marked points on its boundary and $n$ two-sided accumulation points of the marked points; see Figure~\ref{fig:inftygon} on page \pageref{fig:inftygon} for $\mathcal{P}_{1,\infty}$ and $\mathcal{P}_{3,\infty}$.
Note that the accumulation points are not marked in this case.
\item $\mathcal{P}_{\overline{n,\infty}}$ denotes the punctured disk with infinite marked points on its boundary and $n$ two-sided accumulation points of the marked points, which are also marked in this case; see Figure~\ref{fig:completed inftygon} on page \pageref{fig:completed inftygon} for $\mathcal{P}_{\overline{1,\infty}}$ and $\mathcal{P}_{\overline{3,\infty}}$.
\end{itemize}

We also introduce quivers of type $\mathbb{D}_{n,\infty}$ and $\mathbb{D}_{\overline{n,\infty}}$ (Definitions~\ref{def:D infinity} and \ref{def:D infinity closed}, respectively), which are defined using thread quivers (Definition~\ref{def:thread quiver}).

The quiver $\mathbb{D}_{n,\infty}$ can be thought of as having $i$ ``accumulation points'' indicated by the open circles below.
\begin{displaymath}
    \xymatrix@C=2ex@R=3ex{
    \bullet \\
    \bullet \ar[u] \ar[d] &
    \bullet \ar[l] &
    \cdots \circ \cdots \ar[l] &
    \bullet \ar[l] & \bullet \ar[l] & \bullet \ar[l] &
    \cdots \ar[l] & \cdots  \circ \cdots &
    \bullet \ar[l] & \bullet \ar[l] & \bullet \ar[l] &
    \cdots \circ \cdots \ar[l] &
    \bullet \ar[l] & \bullet \ar[l] \\
    \bullet
    }
\end{displaymath}
For the quiver $\mathbb{D}_{\overline{n,\infty}}$, we add the vertices corresponding to the accumulation points.
\begin{displaymath}
    \xymatrix@C=2ex@R=3ex{
    \bullet \\
    \bullet \ar[u] \ar[d] &
    \bullet \ar[l] &
    \cdots \bullet \cdots \ar[l] &
    \bullet \ar[l] & \bullet \ar[l] & \bullet \ar[l] &
    \cdots \ar[l] & \cdots \bullet \cdots &
    \bullet \ar[l] & \bullet \ar[l] & \bullet \ar[l] &
    \cdots \bullet \cdots \ar[l] &
    \bullet \ar[l] & \bullet \ar[l] \\
    \bullet
    }
\end{displaymath}

\begin{theorem}[Theorem~\ref{thm:main_result A body}]\label{thm:main_result A}
    There is a family of infinite type $\mathbb{D}$ cluster categories $\{\mathcal{C}(\mathbb{D}_{n,\infty}) \mid n\in\NN_{>0} \}$.
    Furthermore, for each $\mathcal{C}(\mathbb{D}_{n,\infty})$, the combinatorial data of clusters and mutation is encoded in $\mathcal{P}_{n,\infty}$.
\end{theorem}

The requirement that every element in a cluster be mutable will fail for those elements corresponding to a limiting arc in the triangulation.
(See, for example, \cite{paquette2021completions,igusa2022continuous}.)
Thus, for $\mathbb{D}_{\overline{n,\infty}}$, we only obtain a weak cluster category that induces a cluster theory (Definitions~\ref{def:weak cluster category} and \ref{def:cluster theory}, respectively).
\begin{theorem}[Theorem~\ref{thm:main_result B body}]\label{thm:main_result B}
    There is a family of infinite type $\mathbb{D}$ weak cluster categories $\{\mathcal{C}(\mathbb{D}_{\overline{n,\infty}}) \mid n\in\NN_{>0} \}$ with cluster theories.
    Furthermore, for each $\mathcal{C}(\mathbb{D}_{\overline{n,\infty}})$, the combinatorial data of clusters and mutation is encoded in $\mathcal{P}_{\overline{n,\infty}}$. 
\end{theorem}

\medskip
\noindent{\bf 
Structure of the paper.} The paper is structured as follows. In Section~\ref{sec:cluster_category}, we fix our notation and provide the necessary background from the theory of cluster algebras and cluster categories. We also briefly recall the notion of tilting modules and the relations with cluster algebras. In Section~\ref{sec:infinite}, we provide an overview of the existing cluster categories of infinite type in our scope. Section~\ref{sec:main} contains our main results, in particular the proofs of Theorems~\ref{thm:main_result A body} and
\ref{thm:main_result B body}. 

\medskip

Throughout the paper we assume that $\kk$ is an  algebraically closed field.

\section{Cluster categories}\label{sec:cluster_category}
In this section we review cluster categories and related results.
We begin with cluster algebras and then move on to cluster categories (Sections~\ref{sec:cluster algebras} and \ref{section: cluster category}, respectively).
Afterwards we discuss tilting theory (Section~\ref{sec:tilting theory}).

\subsection{Cluster algebras}\label{sec:cluster algebras}
Cluster algebras are a class of commutative rings introduced by Fomin and Zelevinsky \cite{fomin2002cluster} in the study of Lusztig's total positivity \cite{lusztig2010introduction, lusztig1994total, fomin1999double}. To define a cluster algebra, consider the field of rational functions $\mathbb{F} = \QQ(u_1,\dots,u_n)$ together with a transcendence basis $\mathbf{x} \subseteq \mathbb{F}$ over $\QQ$. Let $B = (b_{xy})_{x,y\in \mathbf{x}}$ be an integer matrix with rows and columns indexed by $\mathbf{x}$ such that for all $x,y \in \mathbf{x}$, $b_{xy}=0$ if and only if $b_{yx}=0$, $b_{xy}>0$ if and only if $b_{yx}<0$, and $b_{xx}=0$. A pair $(\mathbf{x},B)$ is called a \emph{seed}
that must satisfy the exchange
property. That is given a seed $(\mathbf{x},B)$ and $z \in \mathbf{x}$, there is a unique choice of $z'\in\mathbb{F}$, hence a unique choice of another cluster $\mathbf{x}'$, such that $\mathbf{x}' = \mathbf{x}\cup\{z'\}\backslash \{z\}$.
The \emph{mutation} of the matrix $B$ in direction $z$ is defined as the matrix $B' = (b'_{xy})_{x,y\in \mathbf{x}'}$ where $b'_{xy} = -b_{xy}$ if $x=z'$ or $y=z'$, and $b'_{xy} =b_{xy}+ \frac{1}{2}\left( |b_{xz}|b_{zy}+b_{xz}|b_{zy}|\right)$, otherwise.
The pair $(\mathbf{x}',B')$ is called \emph{mutation} of the seed $(\mathbf{x},B)$ in direction $z$. Let $\mathcal{S}$ be the set of all the seeds obtained by iterated mutations of $(\mathbf{x},B)$ and let $\mathcal{X}$ be the union of the transcendence bases appearing in the seeds in $\mathcal{S}$. These bases are called \emph{clusters} and the variables in $\mathcal{X}$ are called \emph{cluster variables}. The \emph{cluster algebra} $\mathcal{A}(\mathbf{x},B)$ is the subring of $\mathbb{F}$ generated by $\mathcal{X}$. Up to algebra isomorphism, the cluster algebra does not depend on the choice of the transcendence basis $\mathbf{x}$, hence we denote it by $\mathcal{A}_B$.

When the matrix $B$ is skew-symmetric, we can construct a quiver associated with the seed $(\mathbf{x},B)$ with vertices corresponding to elements in $\mathbf{x}$ and $b_{xy}$ arrows from $x$ to $y$ whenever $b_{xy}>0$. A cluster algebra $\mathcal{A}_B$ is said to be \emph{of finite type} if $\mathcal{X}$ is a finite set. In \cite{fomin2002clusterfinitetype}, it is shown that cluster algebras of finite type are precisely those for which there exist a seed whose corresponding quiver is of Dynkin type.

\medskip

\noindent{\bf Dynkin diagrams.} Dynkin diagrams are a collection of directed graphs used to classify root systems of simple Lie algebras corresponding to the Lie groups and $\kk$-species with finitely-many isomorphism classes of indecomposable representations. Cluster algebras of finite type are specifically those whose cluster variables correspond to the root system of a Dynkin diagram. As an example, see the following two inﬁnite families of Dynkin diagrams, which are also known as simply-laced Dynkin diagrams.

 \begin{align*}
      \vspace{2cm}   \mathbb{A}_n \quad \xymatrix@C=2ex@R=3ex{
    \bullet &
    \bullet \ar[l] &
    \cdots\ar[l] &\bullet\ar[l] &\bullet \ar[l] &
   \bullet\ar[l]  &\bullet\ar[l] & 
    }
        &&
         \mathbb{D}_n \quad \xymatrix@C=2ex@R=3ex{
    \bullet \\
    \bullet \ar[u] \ar[d] &
    \bullet \ar[l] &
    \cdots\ar[l] &\bullet\ar[l] &\bullet \ar[l] &
   \bullet\ar[l]  &\bullet\ar[l] & 
   \\
  \bullet}
    \end{align*}

\subsection{Cluster categories}\label{section: cluster category}
In this section, we quickly recall some classical results on cluster categories, in particular path and mesh categories, and we refer to \cite{buan2006tilting, buan2007cluster, buan2009cluster} for proofs and more details. Cluster categories, are introduced in \cite{buan2006tilting} to provide a combinatorial model for cluster algebras. The results of this section may be stated for any finite-dimensional hereditary algebra over an algebraically closed field; see e.g.~\cite{buan2006tilting, buan2007cluster}.
However, we only state them for a path algebra, as that is sufficient for our purposes.

Let $\Bbbk Q$ be a finite-dimensional path algebra over a field $\kk$ and denote by $\mD = D^b(\Bbbk Q)$ the bounded derived category of finitely-generated $\Bbbk Q$-modules with shift functor $[1]$. If $\Bbbk Q$ is of finite representation type,
then the category $\mD$ depends only on the underlying undirected graph $\Delta$ of the quiver of $\Bbbk Q$ and $\Delta$ is a simply-laced quiver, i.e.~there is at most one edge between any pair of vertices. 
In this case there is a combinatorial description that uses the theory of translation quivers, introduced in \cite{riedtmann1980algebren}.

\medskip
\noindent{\bf Stable translation quivers.} A \emph{stable translation quiver} is a pair $(\Gamma,\tau)$ with a quiver $\Gamma=(\Gamma_0,\Gamma_1)$ without loops, where $\Gamma_0$ is the set of vertices and $\Gamma_1$ is the set of arrows, and $\tau$ is a bijection on $\Gamma_0$ called \emph{translation} such that for any $x,y \in \Gamma_0$, the number of arrows from $y$ to $x$ is equal to the number of arrows from $\tau x$ to $y$. Given a stable translation quiver $(\Gamma,\tau)$, there exists a bijection $\sigma:\Gamma_1 \to \Gamma_1$ such that every arrow $\alpha:y \to x \in \Gamma_1$ is mapped to the arrow $\sigma(\alpha): \tau x \to y$. Such a bijection is called \emph{polarization}, which is unique if $\Gamma$ has no multiple edges. 

\medskip

\noindent{\bf Path algebras.} Path algebras are important families of algebras associated with quivers. The basis elements of the associated path algebra $\mathbb{k}Q$ of a given finite quiver $Q$ are corresponding to the paths in $Q$. For example, the basis elements of the quivers $\mathbb{D}_4$ and $\mathbb{D}_5$ below are:
\begin{align*}
  \mathcal{B}(\mathcal{P}(\mathbb{D}_4))&=\{e_1,e_2,e_3,e_4,\alpha,\beta,\gamma, \gamma\alpha, \gamma\beta\}
\\
  \mathcal{B}(\mathcal{P}(\mathbb{D}_5))&=\{e_1,e_2,e_3,e_4,e_5,\alpha,\beta,\gamma,\sigma, \gamma\alpha, \gamma\beta, \sigma\gamma,\sigma\gamma\alpha, \sigma\gamma\beta\}.
\end{align*}
Here $e_i$ represents the trivial path on the vertex $i$. Then
the multiplication in $\mathbb{k}Q$ is induced by the composition of paths. For example, $\gamma\cdot \alpha=\gamma\alpha$,  $\beta\cdot \alpha=0$, $\gamma\cdot e_4=e_4$
and  $e_3\cdot \gamma=0$.
 \begin{align*}
        \xymatrix@C=2ex@R=3ex{1 \\ 3\ar[u]^{\alpha} \ar[d]_{\beta} & 4  \ar[l]_{\gamma}  \\ 2}
        &&
         \xymatrix@C=2ex@R=3ex{1 \\ 3\ar[u]^{\alpha} \ar[d]_{\beta} & 4 \ar[l]_{\gamma} & 5 \ar[l]_{\sigma}  \\ 2}
    \end{align*}

\medskip

\noindent{\bf Auslander-Reiten (AR) quivers.} For any finite-dimensional $\mathbb{k}$-algebra, for instance the path algebras of Dynkin quivers, we have an associated
AR-quiver, whose vertices correspond to the isomorphism classes of the indecomposable modules, and the arrows represent the irreducible maps.
For example, for the path algebra $\mathbb{k}Q$ of the
quiver $\mathbb{D}_5$, we have the following AR-quiver. We denote by $P_i$, $I_i$ and $S_i$ for the projective, injective and simple representation corresponding to vertex $i$, respectively. The modules are described via their composition series and each simple module $S_i$ is written as $i$. For instance, $32$ is an indecomposable module $M$ with a simple submodule $S_3$ and $S_2$ such that $M/S_3 = S_2$.
The module $43^21$ is 1-dimensional at $4$ and $1$, and it is 2-dimensional at $3$.

\begin{figure}[h]
    \centering
    \input{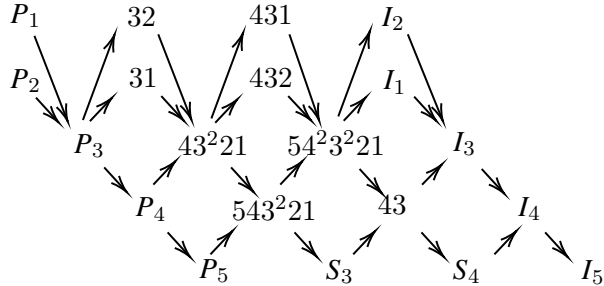}
    \caption{The Auslander-Reiten quiver of the path algebra of type $\mathbb{D}_5$.
}
    \label{AR-quiver D_5}
\end{figure}\medskip
\noindent{\bf Path categories.} The \emph{path category} of $(\Gamma,\tau)$ is the category whose objects are given by the vertices in $\Gamma_0$ and, for any $x,y\in \Gamma_0$, the space of morphisms from $x$ to $y$ is the $\kk$-vector space whose bases are the paths from $x$ to $y$. The \emph{mesh category} $\mM(\Gamma,\tau)$ is the quotient of the path category of $(\Gamma,\tau)$ via the \emph{mesh ideal}, that is, the ideal generated by the \emph{mesh relations}:
\[ m_x = \sum_{\alpha:y\to x}\sigma(\alpha)\alpha. \]
Given a quiver $Q$, we can construct a stable translation quiver $(\ZZ Q , \tau)$ as follows: $(\ZZ Q)_0=\ZZ \times Q_0$ and the number of arrows in $\ZZ Q$ from $(i,x)$ to $(j,y)$ equals the number of arrows in $Q$ from $x$ to $y$ if $i=j$, the number of arrows in $Q$ from $y$ to $x$ if $j=i+1$ and $0$ otherwise. The translation $\tau$ is defined by $\tau((i,x))=(i-1,x)$. The corresponding mesh category is denoted by $\kk(\ZZ Q)$.

\medskip

We recall the following two propositions from \cite{buan2006tilting, buan2007cluster}.
However, we only state them for a path algebra, as that is sufficient for our purposes.
\begin{proposition} \label{prop: properties mD}
Let $Q$ be a simply laced Dynkin quiver. Then the following hold:
\begin{enumerate}[label={\rm (\alph*)}]
    \item for any quiver $Q'$ of the same Dynkin type as $Q$, the derived categories $D^b \kk Q$ and $D^b \kk Q'$ are equivalent;
    \item the Auslander--Reiten quiver of $\mD= D^b \kk Q$ is $\ZZ Q$;
    \item the category $\ind \mD$ is equivalent to the mesh category $\kk(\ZZ Q)$.
\end{enumerate}
\end{proposition}

We are interested in the functor $F = \tau^{-1}[1]$, where $\tau$ is the AR-translation.
It is possible to prove that the category $\mD /F$ is well-defined, the objects are $F$-orbits of objects in $\mD$ and the morphisms are given by
\[ \Hom_{\mD/F}(\Tilde{X}, \Tilde{Y}) = \bigsqcup_{i \in \ZZ} \Hom_\mD(F^i X,Y), \]
where $\Tilde{X}, \Tilde{Y}$ denote the objects in $\mD/F$ corresponding to the objects $X,Y$ in $\mD$.

\begin{proposition}
    The category $\mD/F$ satisfies the following properties:
    \begin{enumerate}[label=(\alph*)]
        \item[{\rm (a)}] $\mD/F$ is a triangulated Krull--Remak--Schmidt category;
        \item[{\rm (b)}] $\mD /F$ has almost split triangles induced by those in $\mD$ and the Auslander--Reiten quiver of $\mD /F$ is $\Gamma(\mD)/\varphi(F)$, where $\varphi(F)$ is the graph automorphism induced by $F$.
    \end{enumerate}
\end{proposition}

\noindent{\bf Cluster category.}
Let $\Bbbk Q$ be a finite-dimensional path algebra, $\mD = D^b(\Bbbk Q)$ the bounded derived category of finitely-generated $\Bbbk Q$-modules with shift functor $[1]$ and $F=\tau^{-1}[1]$. The category $\mC = \mD/F$ is called \emph{cluster category}.

If $\Bbbk Q$ is of finite representation type,
then $\mD$ only depends on the underlying undirected graph
of the quiver of $\Bbbk Q$. Hence, the same is true for the category $\mC$ which leads to a combinatorial description of $\ind \mC$.
More precisely, let $\mS = \ind(({\rm mod}\ \Bbbk Q) \vee (\Bbbk Q[1]) )$ be the set of the indecomposable $\Bbbk Q$-modules, together with the objects $P[1]$, where $P$ is an indecomposable projective $\Bbbk Q$-module. Then $\mS$ is a fundamental domain for the action of $F$ on $\ind \mD$, containing exactly one representative from each $F$-orbit on $\ind \mD$.

\subsection{Tilting theory in cluster categories}\label{sec:tilting theory}

Cluster categories were introduced in \cite{buan2006tilting} to give a new, more general setting for the theory of tilting modules. We will now briefly recall this theory and the relations with cluster algebras.

Let $\Bbbk Q$ be a finite-dimensional path algebra $\mD=D^b(\Bbbk Q)$ and $F=\tau^{-1}[1]$ the functor defined as in \Cref{section: cluster category}. We are interested in studying $\Ext$-configurations, which were introduce in \cite{buan2006tilting} as analogues of $\Hom$-configurations.

\begin{definition}
A subset $\mT$ of non-isomorphic indecomposable objects in $\mD$ or $\mC$ is an \emph{$\Ext$-configuration} if:
    \begin{enumerate}[label=(\roman*)]
        \item $\Ext^1(X,Y)=0$ for all $X$ and $Y$ in $\mT$;
        \item for any indecomposable $Z \not\in \mT$ there is some $X \in \mT$, such that $\Ext^1(X,Z)=0$.
    \end{enumerate}
\end{definition}

The connection between $\Ext$-configurations in $\mD$ and in $\mC$ is given as follows.
\begin{proposition}
Let $\mT$ be an $\Ext$-configuration in $\mD$. Then $\Tilde{\mT} = \{ \Tilde{X} \mid X \in \mT\}$ is an $\Ext$-configuration in $\mC = \mD /F$.
\end{proposition}

We are interested in studying $\Ext$-configurations because they are closely related to the tilting theory of path algebras.
\begin{definition}
    Let $\Bbbk Q$ be a path algebra. A $\Bbbk Q$-module $T$ is called \emph{rigid} if $\Ext(T,T)=0$.
    Additionally, $T$ is said to be a \emph{tilting module} if one of the following equivalent conditions is satisfied:
    \begin{enumerate}[label=(\roman*)]
        \item $T$ is rigid, and there is an exact sequence $0 \to H \to T_0 \to T_1 \to 0$ with $T_0,T_1$ in $\add T$;
        \item $T$ is rigid and has $n$ non-isomorphic indecomposable direct summands, where $n$ is the number of non-isomorphic simple modules;
        \item $T$ is rigid and has a maximum number of non-isomorphic indecomposable direct summands.
    \end{enumerate}
    A tilting module is said to be \emph{basic} if all of its direct summands are non-isomorphic.
    A \emph{tilting set} is a set of non-isomorphic indecomposable objects $\mT$ in $\mD$ or $\mC$ such that it is a rigid set, that is $\Ext^1(T,T')=0$ for every $T,T' \in \mT$, and it is maximal with this property. Finally, an object $T$ in $\mC$ is a \emph{tilting object} if $\Ext^1_\mC(T,T)=0$ and $T$ has a maximal number of non-isomorphic direct summands.
\end{definition}

In the cluster category $\mC$, the following relation between $\Ext$-configurations and tilting sets holds. Note that in general this is not true in the category $\mD$.
\begin{proposition}
    Let $\Tilde{\mT}$ be a set of non-isomorphic objects in $\ind \mC$. Then $\Tilde{\mT}$ is a tilting set if and only if it is an $\Ext$-configuration.
\end{proposition}

Moreover, there is a clear relation between tilting sets, and hence $\Ext$-configurations, in $\mC$ and basic tilting modules over some hereditary algebra equivalent to $H$.

\begin{theorem}
Let $Q$ be a quiver whose path algebra $\Bbbk Q$ has $n$ simple representations.
Let $T$ be a basic tilting object in $\mC = D^b(\Bbbk Q) /F$.
Then:
\begin{enumerate}[label={\rm(\alph*)}]
    \item $T$ is induced by a basic tilting module over a hereditary algebra $H'$, derived equivalent to $\mathbb{k}Q$;
    \item  $T$ has $n$ indecomposable direct summands.
\end{enumerate}
Moreover, any basic tilting module over a hereditary algebra $H$ induces a basic tilting object for $\mC = D^b(\mathbb{k}Q) /F$.
\end{theorem}

If $\Bbbk Q$ is the path algebra of a simply laced quiver $Q$ of Dynkin type, with underlying graph $\Delta$, and $\kk$ is an algebraically closed field, we want to describe the relation between the cluster algebra $\mA = \mA(\Delta)$ and the cluster category $\mC = D^b(\Bbbk Q) /F$. In particular, there is a bijective correspondence between the clusters and the basic tilting objects in $\mC$. 
This is a consequence of the more general result conjectured in \cite{buan2006tilting} and proved in \cite{buan2007cluster}, stating that if $H$ is a finite-dimensional hereditary algebra over an algebraically closed field $\kk$, then there is a bijection between the set of cluster variables of the cluster algebra $\mA$ and the set of indecomposable rigid objects in $\mC$.

\begin{theorem}\label{thm:path_algebra}
Let $Q$ be a quiver with the finite-dimensional path algebra $\mathbb{k}Q$ 
and associated cluster category $\mC$ and cluster algebra $\mA$. Then there is a bijection between the set of isomorphism classes of indecomposable rigid objects in $\mC$ and the set of cluster variables in $\mA$. This induces bijections between clusters and tilting objects.
\end{theorem}

Given the existence of such a bijection, we can study the existence of a mutation of tilting modules by studying the corresponding operation in the cluster setting, which is the cluster mutation. We can formalize this by recalling some definitions from \cite{buan2006tilting}.

\begin{definition}
Let $Q$ be a quiver with the finite-dimensional path algebra $\Bbbk Q$.
A $\Bbbk Q$-module $\Bar{T}$ is said to be an \emph{almost complete basic tilting module} if it is basic, rigid, and has $n-1$ indecomposable direct summands. Then there exists an indecomposable module $M$ such that $\Bar{T} \bigsqcup M$ is a basic tilting module and $M$ is called \emph{complement} to $\Bar{T}$.
\end{definition}

We note that $\Bar{T}$ can be completed to a basic tilting module in at most two different ways \cite{riedtmann1990open, unger1990schur}. Moreover it can be done in exactly two ways if and only if $\Bar{T}$ is sincere \cite{happel1989almost}, that is each simple module appears as a composition factor of $\Bar{T}$.
We can extend this notion to the cluster category $\mC$ as follows.
\begin{definition}
    Let $Q$ be a quiver with the finite-dimensional path algebra $\mathbb{k}Q$ and associated cluster category $\mC$.
    A basic rigid object $\Bar{T}$ in $\mC$ is an \emph{almost complete basic tilting object} if there is an indecomposable object $M$ in $\mC$ such that $\Bar{T}\bigsqcup M$ is a basic tilting object.
\end{definition}
These objects in the cluster category $\mC$ observe a more regular behaviour, as they satisfy the following result.

\begin{theorem}
Let $Q$ be a quiver with the finite-dimensional path algebra $\mathbb{k}Q$ and associated cluster category $\mC$. 
Let $\Bar{T}$ an almost complete basic tilting object in $\mC$.
Then $\Bar{T}$ can be completed to a basic tilting object in $\mC$ in exactly two different ways.
\end{theorem}

It is interesting to study how the two complements of an almost complete basic tilting object are related in the cluster category $\mC$. We say that two non-isomorphic indecomposable objects in $\mC$ form an \emph{exchange pair} if they are the complements of the same almost complete basic tilting object. 

\medskip

Let $\Bar{T}$ be an almost complete basic tilting object in $\mC$ and let $M$ be a complement. It is possible to prove that $\End_\mC(M)$ is a division ring. Then we can consider the factor ring
\[ D_M = \End_\mC(M) / \rad_\mC(M,M). \]
The following description of exchange pairs hold.
\begin{theorem}\label{thm:exchane_pair}
Let $Q$ be a quiver with the finite-dimensional path algebra $\mathbb{k}Q$ 
and associated cluster category $\mC$.
Two indecomposable objects $M$ and $M^*$ in $\mC$ form an exchange pair if and only if \[\dim_{D_M}\Ext^1_\mC(M,M^*)=1 = \dim_{D_{M^*}}\Ext^1_\mC(M^*,M).\]
\end{theorem}

\begin{definition}\label{def:mutation}
    Let $Q$ be a quiver with the finite-dimensional path algebra $\mathbb{k}Q$ and associated cluster category $\mC$ and cluster algebra $\mA$.
    Let $\bar{T}$ be an almost complete basic tilting module and let $M,M^*$ be the complements.
    The act of replacing $M$ in $\bar{T}\bigsqcup M$ with $M^*$ to obtain $\bar{T}\bigsqcup M^*$ is called \emph{mutation}. 
\end{definition}

For types $\mathbb{A}_n$ and $\mathbb{D}_n$, mutation of of basic tilting objects can be modeled using triangulations of the $(n+3)$-gon or triangulations of the punctured $n$-gon, respectively \cite{caldero2006quivers,schiffler2008geometric}.
See Figures~\ref{fig:mut_A5} and \ref{fig:mut_D5} for an example of this mutation in the case of $\mathbb{A}_5$ and $\mathbb{D}_5$, respectively.

\begin{figure}[h]
\centering
    \tikzset{every picture/.style={line width=0.75pt}} %set default line width to 0.75pt        
\begin{tikzpicture}[x=0.75pt,y=0.75pt,yscale=-1,xscale=1]
%uncomment if require: \path (0,401); %set diagram left start at 0, and has height of 401

%Straight Lines [id:da8874067546743476] 
\draw [color={rgb, 255:red, 208; green, 2; blue, 27 }  ,draw opacity=1 ][line width=1.5]    (378.63,49.54) -- (470.26,87.49) ;
%Straight Lines [id:da7995147591368201] 
\draw [color=red!40!blue  ,draw opacity=1 ][line width=1.5]    (44.67,141.17) -- (134.75,49.54) ;
%Straight Lines [id:da30380288600352123] 
\draw [color={rgb, 255:red, 0; green, 0; blue, 0 }  ,draw opacity=1 ][line width=0.75]    (44.67,141.17) -- (174.26,87.49) ;
%Shape: Regular Polygon [id:dp03090787977143017] 
\draw  [line width=1.5]  (174.26,141.17) -- (136.31,179.13) -- (82.63,179.13) -- (44.67,141.17) -- (44.67,87.49) -- (82.63,49.54) -- (136.31,49.54) -- (174.26,87.49) -- cycle ;
%Shape: Circle [id:dp8851472585076506] 
\draw  [fill={rgb, 255:red, 0; green, 0; blue, 0 }  ,fill opacity=1 ] (131.2,49.54) .. controls (131.2,47.58) and (132.79,45.99) .. (134.75,45.99) .. controls (136.71,45.99) and (138.31,47.58) .. (138.31,49.54) .. controls (138.31,51.5) and (136.71,53.09) .. (134.75,53.09) .. controls (132.79,53.09) and (131.2,51.5) .. (131.2,49.54) -- cycle ;
%Shape: Circle [id:dp33905767332685677] 
\draw  [fill={rgb, 255:red, 0; green, 0; blue, 0 }  ,fill opacity=1 ] (170.71,87.49) .. controls (170.71,85.53) and (172.3,83.94) .. (174.26,83.94) .. controls (176.22,83.94) and (177.81,85.53) .. (177.81,87.49) .. controls (177.81,89.46) and (176.22,91.05) .. (174.26,91.05) .. controls (172.3,91.05) and (170.71,89.46) .. (170.71,87.49) -- cycle ;
%Shape: Circle [id:dp30226916223808153] 
\draw  [fill={rgb, 255:red, 0; green, 0; blue, 0 }  ,fill opacity=1 ] (170.71,141.17) .. controls (170.71,139.21) and (172.3,137.62) .. (174.26,137.62) .. controls (176.22,137.62) and (177.81,139.21) .. (177.81,141.17) .. controls (177.81,143.14) and (176.22,144.73) .. (174.26,144.73) .. controls (172.3,144.73) and (170.71,143.14) .. (170.71,141.17) -- cycle ;
%Shape: Circle [id:dp07754767455643075] 
\draw  [fill={rgb, 255:red, 0; green, 0; blue, 0 }  ,fill opacity=1 ] (132.75,179.13) .. controls (132.75,177.17) and (134.34,175.58) .. (136.31,175.58) .. controls (138.27,175.58) and (139.86,177.17) .. (139.86,179.13) .. controls (139.86,181.09) and (138.27,182.68) .. (136.31,182.68) .. controls (134.34,182.68) and (132.75,181.09) .. (132.75,179.13) -- cycle ;
%Shape: Circle [id:dp8752414552716735] 
\draw  [fill={rgb, 255:red, 0; green, 0; blue, 0 }  ,fill opacity=1 ] (79.07,179.13) .. controls (79.07,177.17) and (80.66,175.58) .. (82.63,175.58) .. controls (84.59,175.58) and (86.18,177.17) .. (86.18,179.13) .. controls (86.18,181.09) and (84.59,182.68) .. (82.63,182.68) .. controls (80.66,182.68) and (79.07,181.09) .. (79.07,179.13) -- cycle ;
%Shape: Circle [id:dp08701357090256423] 
\draw  [fill={rgb, 255:red, 0; green, 0; blue, 0 }  ,fill opacity=1 ] (41.12,141.17) .. controls (41.12,139.21) and (42.71,137.62) .. (44.67,137.62) .. controls (46.63,137.62) and (48.22,139.21) .. (48.22,141.17) .. controls (48.22,143.14) and (46.63,144.73) .. (44.67,144.73) .. controls (42.71,144.73) and (41.12,143.14) .. (41.12,141.17) -- cycle ;
%Shape: Circle [id:dp18625385310898857] 
\draw  [fill={rgb, 255:red, 0; green, 0; blue, 0 }  ,fill opacity=1 ] (41.12,87.49) .. controls (41.12,85.53) and (42.71,83.94) .. (44.67,83.94) .. controls (46.63,83.94) and (48.22,85.53) .. (48.22,87.49) .. controls (48.22,89.46) and (46.63,91.05) .. (44.67,91.05) .. controls (42.71,91.05) and (41.12,89.46) .. (41.12,87.49) -- cycle ;
%Shape: Circle [id:dp8571664298476943] 
\draw  [fill={rgb, 255:red, 0; green, 0; blue, 0 }  ,fill opacity=1 ] (79.07,49.54) .. controls (79.07,47.58) and (80.66,45.99) .. (82.63,45.99) .. controls (84.59,45.99) and (86.18,47.58) .. (86.18,49.54) .. controls (86.18,51.5) and (84.59,53.09) .. (82.63,53.09) .. controls (80.66,53.09) and (79.07,51.5) .. (79.07,49.54) -- cycle ;
%Shape: Regular Polygon [id:dp31167630948361114] 
\draw  [line width=1.5]  (470.26,141.17) -- (432.31,179.13) -- (378.63,179.13) -- (340.67,141.17) -- (340.67,87.49) -- (378.63,49.54) -- (432.31,49.54) -- (470.26,87.49) -- cycle ;
%Shape: Circle [id:dp029477881183088472] 
\draw  [fill={rgb, 255:red, 0; green, 0; blue, 0 }  ,fill opacity=1 ] (427.2,49.54) .. controls (427.2,47.58) and (428.79,45.99) .. (430.75,45.99) .. controls (432.71,45.99) and (434.31,47.58) .. (434.31,49.54) .. controls (434.31,51.5) and (432.71,53.09) .. (430.75,53.09) .. controls (428.79,53.09) and (427.2,51.5) .. (427.2,49.54) -- cycle ;
%Shape: Circle [id:dp34225122875196656] 
\draw  [fill={rgb, 255:red, 0; green, 0; blue, 0 }  ,fill opacity=1 ] (466.71,87.49) .. controls (466.71,85.53) and (468.3,83.94) .. (470.26,83.94) .. controls (472.22,83.94) and (473.81,85.53) .. (473.81,87.49) .. controls (473.81,89.46) and (472.22,91.05) .. (470.26,91.05) .. controls (468.3,91.05) and (466.71,89.46) .. (466.71,87.49) -- cycle ;
%Shape: Circle [id:dp665641378198544] 
\draw  [fill={rgb, 255:red, 0; green, 0; blue, 0 }  ,fill opacity=1 ] (466.71,141.17) .. controls (466.71,139.21) and (468.3,137.62) .. (470.26,137.62) .. controls (472.22,137.62) and (473.81,139.21) .. (473.81,141.17) .. controls (473.81,143.14) and (472.22,144.73) .. (470.26,144.73) .. controls (468.3,144.73) and (466.71,143.14) .. (466.71,141.17) -- cycle ;
%Shape: Circle [id:dp9203650856238262] 
\draw  [fill={rgb, 255:red, 0; green, 0; blue, 0 }  ,fill opacity=1 ] (428.75,179.13) .. controls (428.75,177.17) and (430.34,175.58) .. (432.31,175.58) .. controls (434.27,175.58) and (435.86,177.17) .. (435.86,179.13) .. controls (435.86,181.09) and (434.27,182.68) .. (432.31,182.68) .. controls (430.34,182.68) and (428.75,181.09) .. (428.75,179.13) -- cycle ;
%Shape: Circle [id:dp5364277095838589] 
\draw  [fill={rgb, 255:red, 0; green, 0; blue, 0 }  ,fill opacity=1 ] (375.07,179.13) .. controls (375.07,177.17) and (376.66,175.58) .. (378.63,175.58) .. controls (380.59,175.58) and (382.18,177.17) .. (382.18,179.13) .. controls (382.18,181.09) and (380.59,182.68) .. (378.63,182.68) .. controls (376.66,182.68) and (375.07,181.09) .. (375.07,179.13) -- cycle ;
%Shape: Circle [id:dp5028471706031747] 
\draw  [fill={rgb, 255:red, 0; green, 0; blue, 0 }  ,fill opacity=1 ] (337.12,141.17) .. controls (337.12,139.21) and (338.71,137.62) .. (340.67,137.62) .. controls (342.63,137.62) and (344.22,139.21) .. (344.22,141.17) .. controls (344.22,143.14) and (342.63,144.73) .. (340.67,144.73) .. controls (338.71,144.73) and (337.12,143.14) .. (337.12,141.17) -- cycle ;
%Shape: Circle [id:dp1870621390800632] 
\draw  [fill={rgb, 255:red, 0; green, 0; blue, 0 }  ,fill opacity=1 ] (337.12,87.49) .. controls (337.12,85.53) and (338.71,83.94) .. (340.67,83.94) .. controls (342.63,83.94) and (344.22,85.53) .. (344.22,87.49) .. controls (344.22,89.46) and (342.63,91.05) .. (340.67,91.05) .. controls (338.71,91.05) and (337.12,89.46) .. (337.12,87.49) -- cycle ;
%Shape: Circle [id:dp5112648511285086] 
\draw  [fill={rgb, 255:red, 0; green, 0; blue, 0 }  ,fill opacity=1 ] (375.07,49.54) .. controls (375.07,47.58) and (376.66,45.99) .. (378.63,45.99) .. controls (380.59,45.99) and (382.18,47.58) .. (382.18,49.54) .. controls (382.18,51.5) and (380.59,53.09) .. (378.63,53.09) .. controls (376.66,53.09) and (375.07,51.5) .. (375.07,49.54) -- cycle ;
%Straight Lines [id:da9923393036276609] 
\draw    (82.63,179.13) -- (174.26,87.49) ;
%Straight Lines [id:da8883742631010751] 
\draw    (136.31,179.13) -- (174.26,87.49) ;
%Straight Lines [id:da5522703555662836] 
\draw    (44.67,141.17) -- (82.63,49.54) ;
%Straight Lines [id:da6536749451298244] 
\draw    (340.67,141.17) -- (378.63,49.54) ;
%Straight Lines [id:da16099590215535098] 
\draw    (340.67,141.17) -- (470.26,87.49) ;
%Straight Lines [id:da46277889517144977] 
\draw    (378.63,179.13) -- (470.26,87.49) ;
%Straight Lines [id:da015780187262657952] 
\draw    (432.31,179.13) -- (470.26,87.49) ;
%Straight Lines [id:da005580277833317715] 
\draw    (240,110) -- (278.18,110.48) ;
\draw [shift={(280.18,110.5)}, rotate = 180.71] [color={rgb, 255:red, 0; green, 0; blue, 0 }  ][line width=0.75]    (10.93,-3.29) .. controls (6.95,-1.4) and (3.31,-0.3) .. (0,0) .. controls (3.31,0.3) and (6.95,1.4) .. (10.93,3.29)   ;

% Text Node
\draw (76,185.4) node [anchor=north west][inner sep=0.75pt]    {$1$};
% Text Node
\draw (134,185.88) node [anchor=north west][inner sep=0.75pt]    {$2$};
% Text Node
\draw (183,133.88) node [anchor=north west][inner sep=0.75pt]    {$3$};
% Text Node
\draw (181,76.88) node [anchor=north west][inner sep=0.75pt]    {$4$};
% Text Node
\draw (131,25.4) node [anchor=north west][inner sep=0.75pt]    {$5$};
% Text Node
\draw (75,25.4) node [anchor=north west][inner sep=0.75pt]    {$6$};
% Text Node
\draw (24,76.4) node [anchor=north west][inner sep=0.75pt]    {$7$};
% Text Node
\draw (24,134.4) node [anchor=north west][inner sep=0.75pt]    {$8$};
% Text Node
\draw (372,185.4) node [anchor=north west][inner sep=0.75pt]    {$1$};
% Text Node
\draw (430,185.88) node [anchor=north west][inner sep=0.75pt]    {$2$};
% Text Node
\draw (479,133.88) node [anchor=north west][inner sep=0.75pt]    {$3$};
% Text Node
\draw (477,76.88) node [anchor=north west][inner sep=0.75pt]    {$4$};
% Text Node
\draw (427,25.4) node [anchor=north west][inner sep=0.75pt]    {$5$};
% Text Node
\draw (371,25.4) node [anchor=north west][inner sep=0.75pt]    {$6$};
% Text Node
\draw (320,76.4) node [anchor=north west][inner sep=0.75pt]    {$7$};
% Text Node
\draw (320,134.4) node [anchor=north west][inner sep=0.75pt]    {$8$};
\end{tikzpicture}
    \caption{Mutation of a triangulation on the octagon.
    The black diagonals correspond to indecomposable summands of an almost complete basic tilting object $\bar{T}$.
    The purple and red diagonals correspond to complements $M$ and $M^*$, respectively.
    We replace the purple edge with the red edge, mutating the triangulation, which corresponds to replacing $M$ with $M^*$.
    Every mutation of triangulations of the octagon models a cluster mutation for the cluster algebra and cluster category of type $\mathbb{A}_5$.}\label{fig:mut_A5}
\end{figure}
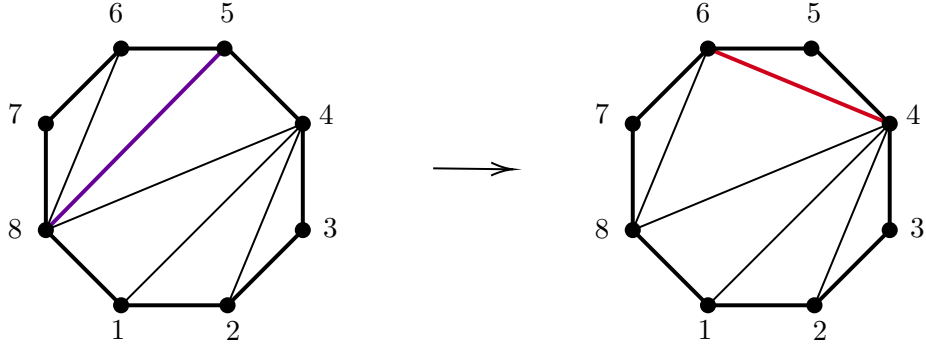

\begin{figure}[h]
\centering
    \tikzset{every picture/.style={line width=0.75pt}} %set default line width to 0.75pt        

\begin{tikzpicture}[x=0.75pt,y=0.75pt,yscale=-1,xscale=1]
%uncomment if require: \path (0,393); %set diagram left start at 0, and has height of 393

%Straight Lines [id:da279007745849445] 
\draw [color={rgb, 255:red, 208; green, 2; blue, 27 }  ,draw opacity=1 ][line width=1.5]    (405,141) -- (473.57,121.37) ;
%Shape: Regular Polygon [id:dp07591115263626991] 
\draw  [line width=1.5]  (112,73) -- (178.57,121.37) -- (153.14,199.63) -- (70.86,199.63) -- (45.43,121.37) -- cycle ;
%Shape: Circle [id:dp22064190448655663] 
\draw  [fill={rgb, 255:red, 0; green, 0; blue, 0 }  ,fill opacity=1 ] (149.59,199.63) .. controls (149.59,197.67) and (151.18,196.08) .. (153.14,196.08) .. controls (155.11,196.08) and (156.7,197.67) .. (156.7,199.63) .. controls (156.7,201.59) and (155.11,203.18) .. (153.14,203.18) .. controls (151.18,203.18) and (149.59,201.59) .. (149.59,199.63) -- cycle ;
%Shape: Circle [id:dp19995400882943215] 
\draw  [fill={rgb, 255:red, 0; green, 0; blue, 0 }  ,fill opacity=1 ] (175.02,121.37) .. controls (175.02,119.41) and (176.61,117.82) .. (178.57,117.82) .. controls (180.54,117.82) and (182.13,119.41) .. (182.13,121.37) .. controls (182.13,123.33) and (180.54,124.92) .. (178.57,124.92) .. controls (176.61,124.92) and (175.02,123.33) .. (175.02,121.37) -- cycle ;
%Shape: Circle [id:dp6125124315988184] 
\draw  [fill={rgb, 255:red, 0; green, 0; blue, 0 }  ,fill opacity=1 ] (108.45,73) .. controls (108.45,71.04) and (110.04,69.45) .. (112,69.45) .. controls (113.96,69.45) and (115.55,71.04) .. (115.55,73) .. controls (115.55,74.96) and (113.96,76.55) .. (112,76.55) .. controls (110.04,76.55) and (108.45,74.96) .. (108.45,73) -- cycle ;
%Shape: Circle [id:dp5272290218769778] 
\draw  [fill={rgb, 255:red, 0; green, 0; blue, 0 }  ,fill opacity=1 ] (106.45,141) .. controls (106.45,139.04) and (108.04,137.45) .. (110,137.45) .. controls (111.96,137.45) and (113.55,139.04) .. (113.55,141) .. controls (113.55,142.96) and (111.96,144.55) .. (110,144.55) .. controls (108.04,144.55) and (106.45,142.96) .. (106.45,141) -- cycle ;
%Straight Lines [id:da6309387707050231] 
\draw    (110,141) -- (71.41,199.63) ;
%Straight Lines [id:da2892804599898312] 
\draw    (45.43,121.37) -- (110,141) ;
%Curve Lines [id:da6545019986331806] 
\draw [color=red!40!blue  ,draw opacity=1 ][line width=1.5]    (45.43,121.37) .. controls (146,117.97) and (167,123.97) .. (71.41,199.63) ;
%Curve Lines [id:da9953273069999807] 
\draw    (45.43,121.37) .. controls (85.43,91.37) and (159,104.97) .. (178.57,121.37) ;
%Curve Lines [id:da5425475160143781] 
\draw    (70.86,199.63) .. controls (112,199.97) and (171,148.97) .. (178.57,121.37) ;
%Shape: Circle [id:dp8827983940462967] 
\draw  [fill={rgb, 255:red, 0; green, 0; blue, 0 }  ,fill opacity=1 ] (41.87,121.37) .. controls (41.87,119.41) and (43.46,117.82) .. (45.43,117.82) .. controls (47.39,117.82) and (48.98,119.41) .. (48.98,121.37) .. controls (48.98,123.33) and (47.39,124.92) .. (45.43,124.92) .. controls (43.46,124.92) and (41.87,123.33) .. (41.87,121.37) -- cycle ;
%Shape: Circle [id:dp7848543729816473] 
\draw  [fill={rgb, 255:red, 0; green, 0; blue, 0 }  ,fill opacity=1 ] (67.86,199.63) .. controls (67.86,197.67) and (69.45,196.08) .. (71.41,196.08) .. controls (73.37,196.08) and (74.96,197.67) .. (74.96,199.63) .. controls (74.96,201.59) and (73.37,203.18) .. (71.41,203.18) .. controls (69.45,203.18) and (67.86,201.59) .. (67.86,199.63) -- cycle ;
%Shape: Regular Polygon [id:dp4417216184061076] 
\draw  [line width=1.5]  (407,73) -- (473.57,121.37) -- (448.14,199.63) -- (365.86,199.63) -- (340.43,121.37) -- cycle ;
%Shape: Circle [id:dp8883649231408572] 
\draw  [fill={rgb, 255:red, 0; green, 0; blue, 0 }  ,fill opacity=1 ] (444.59,199.63) .. controls (444.59,197.67) and (446.18,196.08) .. (448.14,196.08) .. controls (450.11,196.08) and (451.7,197.67) .. (451.7,199.63) .. controls (451.7,201.59) and (450.11,203.18) .. (448.14,203.18) .. controls (446.18,203.18) and (444.59,201.59) .. (444.59,199.63) -- cycle ;
%Shape: Circle [id:dp9545972444423483] 
\draw  [fill={rgb, 255:red, 0; green, 0; blue, 0 }  ,fill opacity=1 ] (470.02,121.37) .. controls (470.02,119.41) and (471.61,117.82) .. (473.57,117.82) .. controls (475.54,117.82) and (477.13,119.41) .. (477.13,121.37) .. controls (477.13,123.33) and (475.54,124.92) .. (473.57,124.92) .. controls (471.61,124.92) and (470.02,123.33) .. (470.02,121.37) -- cycle ;
%Shape: Circle [id:dp09463766224973602] 
\draw  [fill={rgb, 255:red, 0; green, 0; blue, 0 }  ,fill opacity=1 ] (403.45,73) .. controls (403.45,71.04) and (405.04,69.45) .. (407,69.45) .. controls (408.96,69.45) and (410.55,71.04) .. (410.55,73) .. controls (410.55,74.96) and (408.96,76.55) .. (407,76.55) .. controls (405.04,76.55) and (403.45,74.96) .. (403.45,73) -- cycle ;
%Shape: Circle [id:dp12903358494192552] 
\draw  [fill={rgb, 255:red, 0; green, 0; blue, 0 }  ,fill opacity=1 ] (401.45,141) .. controls (401.45,139.04) and (403.04,137.45) .. (405,137.45) .. controls (406.96,137.45) and (408.55,139.04) .. (408.55,141) .. controls (408.55,142.96) and (406.96,144.55) .. (405,144.55) .. controls (403.04,144.55) and (401.45,142.96) .. (401.45,141) -- cycle ;
%Straight Lines [id:da7831431810432959] 
\draw    (405,141) -- (366.41,199.63) ;
%Straight Lines [id:da9255823871070105] 
\draw    (340.43,121.37) -- (405,141) ;
%Curve Lines [id:da8666527263249268] 
\draw    (340.43,121.37) .. controls (380.43,91.37) and (454,104.97) .. (473.57,121.37) ;
%Curve Lines [id:da5322997864167505] 
\draw    (365.86,199.63) .. controls (407,199.97) and (466,148.97) .. (473.57,121.37) ;
%Shape: Circle [id:dp6540016508161616] 
\draw  [fill={rgb, 255:red, 0; green, 0; blue, 0 }  ,fill opacity=1 ] (336.87,121.37) .. controls (336.87,119.41) and (338.46,117.82) .. (340.43,117.82) .. controls (342.39,117.82) and (343.98,119.41) .. (343.98,121.37) .. controls (343.98,123.33) and (342.39,124.92) .. (340.43,124.92) .. controls (338.46,124.92) and (336.87,123.33) .. (336.87,121.37) -- cycle ;
%Shape: Circle [id:dp8658491112469081] 
\draw  [fill={rgb, 255:red, 0; green, 0; blue, 0 }  ,fill opacity=1 ] (362.86,199.63) .. controls (362.86,197.67) and (364.45,196.08) .. (366.41,196.08) .. controls (368.37,196.08) and (369.96,197.67) .. (369.96,199.63) .. controls (369.96,201.59) and (368.37,203.18) .. (366.41,203.18) .. controls (364.45,203.18) and (362.86,201.59) .. (362.86,199.63) -- cycle ;
%Straight Lines [id:da9731429049824989] 
\draw    (239,119.23) -- (278,119.23) ;
\draw [shift={(280,119.23)}, rotate = 180] [color={rgb, 255:red, 0; green, 0; blue, 0 }  ][line width=0.75]    (10.93,-3.29) .. controls (6.95,-1.4) and (3.31,-0.3) .. (0,0) .. controls (3.31,0.3) and (6.95,1.4) .. (10.93,3.29)   ;

% Text Node
\draw (62,205.4) node [anchor=north west][inner sep=0.75pt]    {$1$};
% Text Node
\draw (150,205.4) node [anchor=north west][inner sep=0.75pt]    {$2$};
% Text Node
\draw (184,112.4) node [anchor=north west][inner sep=0.75pt]    {$3$};
% Text Node
\draw (105,51.4) node [anchor=north west][inner sep=0.75pt]    {$4$};
% Text Node
\draw (27,112.4) node [anchor=north west][inner sep=0.75pt]    {$5$};
% Text Node
\draw (358,205.4) node [anchor=north west][inner sep=0.75pt]    {$1$};
% Text Node
\draw (445,205.4) node [anchor=north west][inner sep=0.75pt]    {$2$};
% Text Node
\draw (482,112.4) node [anchor=north west][inner sep=0.75pt]    {$3$};
% Text Node
\draw (399,51.4) node [anchor=north west][inner sep=0.75pt]    {$4$};
% Text Node
\draw (324,112.4) node [anchor=north west][inner sep=0.75pt]    {$5$};
\end{tikzpicture}
    \caption{Mutation of a triangulation on the punctured pentagon.
    The black diagonals correspond to indecomposable summands of an almost complete basic tilting object $\bar{T}$.
    The purple and red diagonals correspond to complements $M$ and $M^*$, respectively.
    We replace the purple edge with the red edge, mutating the triangulation, which corresponds to replacing $M$ with $M^*$.
    Every mutation of triangulations on the punctured pentagon models a mutation for the cluster algebra and cluster category of type $\mathbb{D}_5$.}\label{fig:mut_D5}
\end{figure}
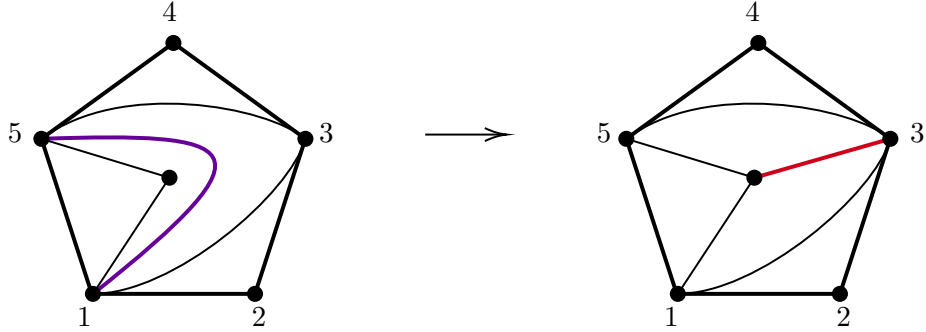

\section{Infinite representations and infinite cluster categories}\label{sec:infinite}
In this section, we move from the finite setting to the infinite.
We start by looking at representations of thread quivers (Section~\ref{sec:thread quivers}).
Then we discuss infinite cluster categories (Section~\ref{sec:infinite  cluster categories}).
We conclude with the further abstractions of weak cluster structures and cluster theories (Sections~\ref{sec:cluster structures} and \ref{sec:cluster theories}, respectively).

\subsection{Pointwise finite-dimensional representations of thread quivers}\label{sec:thread quivers}
We recall the notion of thread quivers from \cite{berg2014threaded}. They are introduced as an (infinite) generalization of quivers. One of their significant properties is that every $\kk$-linear hereditary category with Serre duality and enough projectives is equivalent to the category of finitely presented representations of a thread quiver. 

\medskip
\begin{definition}\label{def:thread quiver}
A thread quiver consists of the following data:
\begin{itemize}
\item 
A quiver $\QQ = (\QQ_0, \QQ_1)$ where $\QQ_0$ is the set of vertices and $\QQ_1$ is the set of arrows.
\item A decomposition $\QQ_1 = \QQ_s \bigsqcup\QQ_t$, where arrows in $\QQ_s$ are called standard arrows, and arrows in $\QQ_t$ are called thread arrows.
\item An associated linearly ordered set $P_t$ (possibly empty) for every thread arrow $t$. 
    \end{itemize}
To every thread arrow $t : x \xrightarrow {P_t}y$ in $\QQ$ we associate a 
linearly ordered poset with a minimal and a maximal element.
Usually, we take a locally discrete set.
This means that every vertex has a predecessor (except the minimum) and successor (except the maximum).
More precisely, either we use $L_t = \NN\cdot(J_t\overrightarrow{\times} \ZZ)\cdot-\NN$, where $J_t \overrightarrow{\times}\ZZ$ is the poset $J_t \times \ZZ$ with the lexicographical ordering, or some finite set.
We will use $L_t$ later.
Moreover, the $L_t$ is considered as a category, and $\kk L_t$ is the associated $\kk$-linear additive category.
\end{definition}
If a subset of a thread arrow has no maximum (or minimum) element, we call the supremum (or infimum) an accumulation point, even if it is not in the linearly ordered poset.

Any quiver $Q$ is a trivial example of a thread quiver where the set $Q_t$ is empty.
We study a family of examples in Section~\ref{sec:D infinity}.

\medskip

Let $\QQ$ be a thread quiver.
By $\Rep_\Bbbk(\QQ)$ we denote the category of functors $\QQ\to \Bbbk$-vec.
By $\Rep_\Bbbk^\text{pwf}(\QQ)$ we denote the full subcategory of $\Rep_\Bbbk(\QQ)$ that factors through $\Bbbk$-vec (finite-dimensional $\Bbbk$-vector spaces).
By $\rep_\Bbbk(\QQ)$ we denote the full subcategory of $\Rep_\Bbbk^\text{pwf}(\QQ)$ whose objects are finitely generated by projective objects.
Finally, by $\rep_\Bbbk^\text{fp}(\QQ)$ we denote the full subcategory of $\rep_\Bbbk(\QQ)$ whose objects are finitely presented.
This means each object in $\rep_\Bbbk^\text{fp}(\QQ)$ has a projective cover that is a finite direct sum of representable projectives whose kernel is the direct sum of finitely many representable projectives.
Each of these subcategories is also a wide subcategory of the category in which it is defined.

\begin{proposition}\label{prop:distinct subcategories}
    Let $\QQ$ be a thread quiver where $\QQ_1$ is finite.
    Then \[\rep_\Bbbk^{\fp}(\QQ) \subseteq \rep_\Bbbk(\QQ) \subsetneq \Rep_\Bbbk^{\pwf}(\QQ) \subsetneq\Rep_\Bbbk(\QQ),\]
    where ``$\mathcal A\subsetneq \mathcal B$'' means ``$\mathcal A$ is a proper subcategory of $\mathcal B$.''
    The inclusion $\rep_\Bbbk^{\fp}(\QQ) \subseteq \rep_\Bbbk(\QQ)$ is strict if and only if $\QQ$ has a thread arrow with an accumulation point.
\end{proposition}
\begin{proof}
    We see that $\Rep_\Bbbk^\text{pwf}(\QQ) \subsetneq\Rep_\Bbbk(\QQ)$ by noting that the object $\bigoplus_{i\in\ZZ} M$, for any other object $M$, is in $\Rep_\Bbbk(\QQ)$ but is not in $\Rep_\Bbbk^\text{pwf}(\QQ)$.
    Next we see that the direct sum of all simple representations $\bigoplus_{j\in\QQ} S_j$ is in $\Rep_\Bbbk^\text{pwf}(\QQ)$ but is not in $\rep_\Bbbk(\QQ)$.
    
    Suppose $\QQ$ has an accumulation point and let $j$ be an accumulation point in $\QQ$.
    Then there is a projective $P_j$, but it is not representable.
    Thus, the object $P_j$ is in $\rep_\Bbbk(\QQ)$ but it is not in $\rep_\Bbbk^\text{fp}(\QQ)$.
    
    If $\QQ$ has no accumulation points, then all projective objects are representable because $\QQ$ is equivalent to a finite quiver.
\end{proof}

\begin{remark}\label{rmk:bounded category}
    There is another wide subcategory between $\rep_\Bbbk(\QQ)$ and $\Rep_\Bbbk^\text{pwf}(\QQ)$ called the bounded subcategory, denoted by $\Rep_\Bbbk^\text{B}(\QQ)$.
    This is a different notion of ``bounded'' from the bounded derived category, hence the capital `B' instead of lowercase `b.'
    For each object $M$ in $\Rep_\Bbbk^\text{B}(\QQ)$ there exists an $n\in\NN$ such that $\dim_\Bbbk M(j) < n$ for all $j\in\QQ$.
If we index the vertices of $\QQ$ (including those in the thread arrows) by $\NN$, then the object $\bigoplus_{j\in\NN_{>0}} (\bigoplus_{k=1}^j S_j)$ is in $\Rep_\Bbbk^\text{pwf}(\QQ)$ but it is not in $\Rep_\Bbbk^\text{B}(\QQ)$.
    While we do not use $\Rep_\Bbbk^\text{B}(\QQ)$ in this paper, the category is interesting in its own right.
\end{remark}

\subsection{Infinite cluster categories}\label{sec:infinite  cluster categories}

\subsubsection{The infinity-gons}

\begin{definition}\label{def:unpunctured infinity-gon}
    Let $D$ be the closed 2-dimensional ball $D=\{(x,y)\in\RR^2\text{ such that } ||(x,y)||\leq 1\}$.
    \begin{itemize}
    \item The $\infty$-gon, denoted $\mathcal{U}_\infty$, is given by $D$ together with a set of marked points $\mathcal{M}$ such that $\mathcal{M}\cap \interior D=\emptyset$ and $\mathcal{M}\cap \partial D\simeq \ZZ$.
    See Figure~\ref{fig:unpuctured infty-gon} (left).
    \item For $n\in\NN_{>0}$ the $(n,\infty)$-gon $\mathcal{U}_{n,\infty}$ is given by $D$ together with a set of marked points $\mathcal{M}$ such that $\mathcal{M}\cap \interior D=\emptyset$ and $M\cap \partial D\simeq \ZZ$ has no one-sided accumulation points and has $n$ two-sided accumulation points $\{a_1,\ldots,a_n\}$ such that $a_i\notin \mathcal{M}$ for every $i=1,\ldots,n$.
    Notice $\mathcal{U}_\infty=\mathcal{U}_{1,\infty}$. See Figure~\ref{fig:unpuctured infty-gon} (right).
    \item The completed $(n,\infty)$-gon, denoted $\mathcal{U}_{\overline{n,\infty}}$ is defined by taking $\mathcal{U}_{n,\infty}$ and including $a_1,\ldots,a_n$ in the set $\mathcal{M}$. See Figure~\ref{fig:unpuctured completed infty-gon}.
    \end{itemize}
\end{definition}

\begin{figure}[h]
    \centering
    \input{U_1infty_with_numbers}
    \caption{On the left $\mathcal{U}_\infty=\mathcal{U}_{1,\infty}$ and on the right $\mathcal{U}_{3,\infty}$ from Definition~\ref{def:unpunctured infinity-gon}.}
    \label{fig:unpuctured infty-gon}
\end{figure}

\begin{figure}[h]
    \centering
    \input{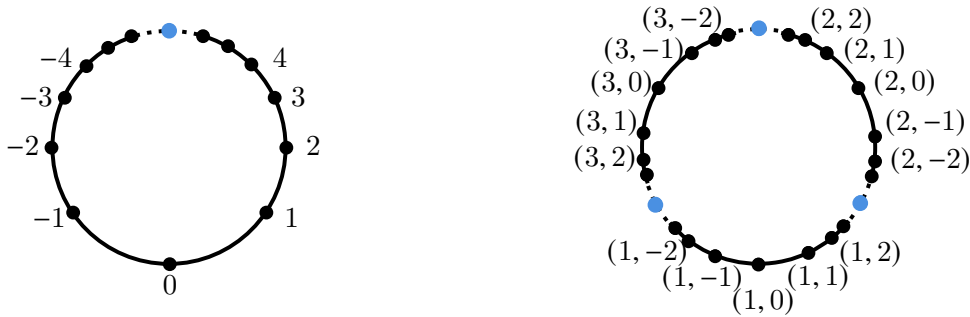}
    \caption{On the left $\mathcal{U}_{\overline{1,\infty}}$ and on the right $\mathcal{U}_{\overline{3,\infty}}$. The blue marked points are two-sided accumulation points. See Definition~\ref{def:unpunctured infinity-gon}.}
    \label{fig:unpuctured completed infty-gon}
\end{figure}

In \cite{holm2012cluster}, Holm and J{\o}rgensen studied a cluster category from the infinity-gon, which generalized type $\mathbb{A}_n$ cluster categories and their connection to polygons.
They proved that, when we omit certain clusters, we obtain a cluster structure.
See Figure~\ref{fig:mut_A_inf} for an example of a mutation in this setting.
\begin{figure}[h]
\centering
    \input{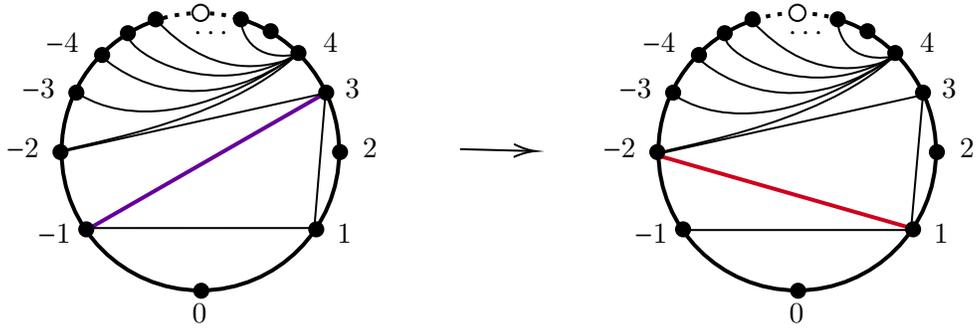}
    \caption{Mutation of a triangulation on the $\infty$-gon.
    We replace the purple edge with the red edge.
    The cluster category of type $\mathbb{A}_\infty$ and the corresponding combinatorial model on the $\infty$-gon was introduced by Holm and J{\o}rgensen.}
    \label{fig:mut_A_inf}
\end{figure}
See Figure~\ref{fig:weak cluster holm-jorgensen} for an example of a cluster that must be omitted.
The cluster in Figure~\ref{fig:weak cluster holm-jorgensen} is not functorially finite and thus fails one of the properties required for a cluster structure.
See \cite{holm2012cluster} for the same example and more discussion.

In \cite{igusa2015cyclic}, Igusa and Todorv generalized Holm and J{\o}rgensen's construction from $\mathcal{U}_\infty=\mathcal{U}_{1,\infty}$ to $\mathcal{U}_{n,\infty}$ for any $n\in\NN_{>0}$.
They did this by considering the marked points of $\mathcal{U}_{n,\infty}$ as a cyclic poset.
Their construction allows many other possibilities but we only emphasize the $(n,\infty)$-gons here.

\subsubsection{The completed infinity-gons}
In \cite{baur2018infinity}, Baur and Graz created a cluster category from the completion of the $\infty$-gon by adding separate points for $-\infty$ and $+\infty$.
Each of their triangulations then have the \emph{generic arc} between $-\infty$ and $+\infty$, since it could cross no other arc.
This may be viewed, in some sense, as a side of the infinite polygon.
However, this completion has two one-sided accumulation points, instead of one two-sided accumulation point.

In \cite{paquette2021completions}, Paquette and Y{\i}ld{\i}r{\i}m created (weak) cluster categories from a completion of the $(n,\infty)$-gon used by Igusa and Todorov in \cite{igusa2015cyclic} (and thus the completion of the $\infty$-gon used by Holm and J{\o}rgensen in \cite{holm2012cluster}).
In their work, Paquette and Y{\i}ld{\i}r{\i}m added one new point for each accumulation point in Igusa and Todorov's work, obtaining what we call $\mathcal{U}_{\overline{n,\infty}}$, by replacing each accumulation point with another copy of $\ZZ$ and then collapsing each new copy of $\ZZ$ to a point.
We take inspiration from this technique later.
When $n=1$, this results in a different completion of the $\infty$-gon from Baur and Graz's work since only one point is added instead of two; the single accumulation point is two-sided.

\subsubsection{Continuous cluster categories}
For completeness we discuss some work on continuous cluster categories.
However, as our intent is to construct discrete cluster categories, we will be brief.

The continuous cluster category, for type $\mathbb{A}$, was introduced and studied by Igusa and Todorov in \cite{igusa2015continuous} and also appears in \cite{igusa2015cyclic}.
This generalizes cluster categories for discrete type $\mathbb{A}_n$.
In particular, they generalized triangulations of polygons to laminations of the hyperbolic pane, which can be thought of as the $\mathbb{R}$-gon, where every point is also an accumulation point.
Igusa, Todorov, and the second author introduced a weak continuous cluster category by starting with all representations of continuous type $\mathbb{A}$ \cite{igusa2022continuous}.
In a weak cluster category, not every element in a cluster need to be mutable.
In \cite{rock2022continuous}, the second author generalized laminations of the hyperbolic plane to a new geometric model for the weak cluster categories in \cite{igusa2022continuous}.

\subsubsection{Infinite cluster categories of type $\mathbb{D}$}

As far as the authors know, there are only two existing infinite cluster categories of type $\mathbb{D}$.
The first is a continuous version introduced by Igusa and Todorov in 2013 \cite{igusa2013continuous}.
Their construction is made using Frobenius categories.

The second is a cluster category of type $\mathbb{D}_\infty$ by Yang in 2017 \cite{yang2017cluster}.
Yang uses an infinite quiver of type $\mathbb{D}$ with a zigzag orientation on the long tail and thus works with modules of finite dimension.
\begin{displaymath}
    \xymatrix@C=3ex@R=4ex{
    \bullet \\
    \bullet \ar[u] \ar[d] \ar[r] & \bullet & \bullet \ar[l] \ar[r] & \bullet & \bullet \ar[l] \ar[r] & \bullet & \bullet \ar[l] \ar[r] & \bullet & \bullet \ar[l] \ar[r] & \bullet & \bullet \ar[l] \ar[r] & \cdots \\
    \bullet
    }
\end{displaymath}
In contrast, our construction uses the straight orientation on the long tail and so we have pointwise finite-dimensional.
Both Yang's construction and our construction consider finitely-presented representations.

\subsection{Weak cluster structures}\label{sec:cluster structures}

Cluster categories can be generalized to cluster structures as shown in \cite{buan2009cluster}. In particular cluster structures preserve some important properties of cluster algebra, such as the unique exchange property and the existence of associated exchange triangles.
However, a cluster structure is too stringent for our requirements.
Instead, we need weak cluster structures.

\medskip

First we recall the definitions of approximations and weakly cluster tilting subcategory.
\begin{definition}\label{def:approximations}
Let $\mC$ be an additive category, $\mathcal{X}$ a full subcategory of $\mC$, and $M$ and object in $\mC$ but not in $\add(\mathcal{X})$.
\begin{itemize}
    \item A left $\add(\mathcal{X})$-approximation of $M$ is a map $f:M\to X$ in $\mC$, where $X$ is in $\add(\mathcal{X})$.
    Furthermore, if $f':M\to X'$ is any other map in $\mC$, where $X'$ is in $\add(\mathcal{X})$, there exists a unique $h:X\to X'$ such that $f'=hf$.
    \item A right $\add(\mathcal{X})$-approximation of $M$ is a map $g:X\to M$ in $\mC$, where $X$ is in $\add(\mathcal{X})$.
    Furthermore, if $g':X'\to M$ is any other map in $\mC$, where $X'$ is in $\add(\mathcal{X})$, there exists a unique $h:X'\to X$ such that $f'=fh$.
\end{itemize}
\end{definition}

The following is the $d=1$ case of a more general definition in \cite{holm2015weak}.

\begin{definition}\label{def:weak cluster tilting}
    Let $\mC$ be a triangulated category and $T$ a full subcategory of $\mC$.
    We say $T$ is \emph{weakly cluster tilting} if, for all objects $X,Y$ in $T$, \[\Ext(X,Y)=0=\Ext(Y,X).\]
\end{definition}

We now present the definition of a weak cluster structure without the notion of coefficients.
We refer the reader to \cite{buan2009cluster} for the version with coefficients.
\begin{definition}\label{def:weak cluster structure}
    Let $\mC$ be a Krull--Remak--Scmidt triangulated category and let $\mathcal{T}$ be a set of weakly cluster tilting subcategories in $\mC$.
    Then $\mathcal{T}$ forms a weak cluster structure if the following hold.
    \begin{enumerate}[label=(\roman*)]
        \item For each weakly cluster tilting subcategory $T\in\mathcal{T}$ and each indecomposable $M\in T$, there exists a unique indecomposable object $M^* \not\cong M$ such that replacing $M$ by $M^*$ gives a new weakly cluster tilting subcategory $T^*\in\mathcal{T}$.
        \item For each indecomposable $M$ in a $T\in\mathcal{T}$, there are triangles $M^* \overset{f}{\to} B \overset{g}{\to} M\to M^*[1]$ and $M \overset{s}{\to} B' \overset{t}{\to} M^*\to M[1]$, where $g$ and $t$ are minimal right $\add(T\backslash\{M\})$-approximations and $f$ and $s$ are minimal left $\add(T\backslash\{M\})$-approximations.
    \end{enumerate}
\end{definition}

We point out that the requirements for a cluster structure are more strict.
In particular, each cluster must be functorially finite.
However, in our case, this will not happen.
We will have cluster such as those in Figure~\ref{fig:weak cluster holm-jorgensen} that are not functorially finite.
\begin{figure}[h]
    \centering
    \input{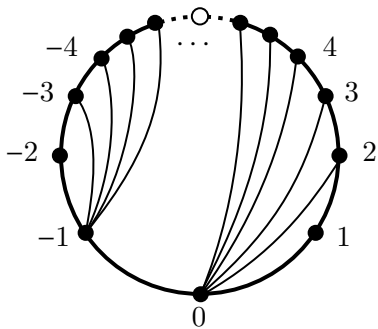}
    \caption{An example of a cluster in Holm and J{\o}rgensen's construction that is not functorially finite in its cluster category.
    Such a cluster cannot be included a cluster structure.
    The lack of functorial finiteness comes from the fact that the two ``fountains'' at $-1$ and $0$ do not coincide.}
    \label{fig:weak cluster holm-jorgensen}
\end{figure}
Thus, we only concern ourselves with weak cluster structures.

\subsection{Cluster theories}\label{sec:cluster theories}
In this section we discuss cluster theories, which generalize cluster categories and weak cluster structures, and embeddings of cluster theories.
These were introduced by Igusa, Todorov, and the second author in \cite{igusa2022continuous} in order to work around the obstruction mentioned at the end of this section.

\begin{definition}\label{def:cluster theory}
	Let $\mathcal{C}$ be a skeletally small, Krull--Remak--Schmidt, additive category and $\mathbf P:\ind(\mathcal{C})\times \ind(\mathcal{C}) \to \{0,1\}$ a pairwise compatibility condition. 
	Suppose that for every maximally $\mathbf P$-compatible set $T\subset \ind(\mathcal{C})$ and every $X\in T$ there exists zero or one $Y\neq X$ in $\ind(\mathcal{C})$ such that $(T\setminus\{X\})\cup\{Y\}$ is a $\mathbf P$-compatible set.
	\begin{itemize}
		\item We call maximally $\mathbf P$-compatible sets \emph{$\mathbf P$-clusters}.
		\item The function $\mu:T\to (T\setminus\{X\})\cup\{Y\}$ that takes $X\mapsto Y$ and $Z\mapsto Z$ whenever $Z\neq X$ is called a \emph{$\mathbf P$-mutation at $X$} or just a \emph{$\mathbf P$-mutation}.
		\item The groupoid whose objects are $\mathbf P$-clusters and whose morphisms are generated by $\mathbf P$-mutations (and identity functions) is called the \emph{$\mathbf P$-cluster theory of $\mathcal{C}$} and is denoted $\mathscr{T}_{\mathbf P}(\mathcal{C})$.
		\item We denote by $I_{\mathbf P,\mathcal{C}}$ the functor $\mathscr{T}_{\mathbf P}(\mathcal{C})\to \mathcal{S}ets$ that takes each set to itself and each function to itself.
	\end{itemize}
\end{definition}

All cluster categories and cluster structures in the sense of \cite{buan2006tilting, buan2009cluster} yield a cluster theory by using $\Ext$-orthogonality as the pairwise compatibility condition.
The cluster category in \cite{igusa2015continuous} yields a cluster theory which we denote by $\mathscr{T}_{\mathbf{N}_{\mathbb{R}}}(\mathcal{C}_\pi)$.
In particular, cluster theories allow us to consider all the triangulations of $\mathcal{U}_{\overline{n,\infty}}$ (Definition~\ref{def:unpunctured infinity-gon}), even those that have arcs we cannot mutate (see \cite{paquette2021completions}).

\begin{definition}
	Let $\mathscr{T}_{\mathbf P}(\mathcal{C})$ and $\mathscr{T}_{\mathbf Q}(\mathcal{D})$ be cluster theories.
	An \emph{embedding of cluster theories} is a pair $(F,\eta)$ where
	\begin{itemize}
		\item $F$ is a functor $F:\mathscr{T}_{\mathbf P}(\mathcal{C})\to \mathscr{T}_{\mathbf Q}(\mathcal{D})$ such that $F$ takes $\mathbf P$-mutations to $\mathbf Q$-mutations and is injective on both objects and morphisms
		\item $\eta$ is a natural transformation $\eta:I_{\mathbf{P},\mathcal{C}}\to I_{\mathbf{Q},\mathcal{D}}\circ F$ such that each component morphism $\eta_T:I_{\mathbf{P},\mathcal{C}}(T)\to I_{\mathbf{Q},\mathcal{D}}\circ F(T)$ is injective.
	\end{itemize}
\end{definition}

The weak cluster category in \cite{igusa2022continuous} yields a cluster theory which we denote by $\mathscr{T}_{\mathbf{E}}(\mathcal{C}(A_{\mathbb{R},S}))$.
This leads us to the following definition.
\begin{definition}\label{def:weak cluster category}
    Let $\mathcal{D}$ be a triangulated category and $\mathcal{C}$ a Krull--Remak--Schmidt, triangulated, orbit category of $\mathcal{D}$.
    Let $\mathbf P$ be a pairwise compatibility condition on the indecomposable objects in $\mathcal{C}$.
    If $\mathbf P$ induces a cluster theory then we call $\mathcal{C}$ a \emph{weak cluster category}.
\end{definition}

\noindent{\bf The Obstruction.}
Cluster theories and embeddings of cluster theories were introduced to work around the following obstruction.
As discussed in \cite[Remark 5.3.9]{igusa2022continuous}, there is no functorial way to embed the continuous cluster category from \cite{igusa2015continuous} into the weak cluster category from \cite{igusa2022continuous}, but the combinatorics appear to be compatible.
However, by \cite[Theorem 5.3.8]{igusa2022continuous} there is an embedding of cluster theories $\mathscr{T}_{\mathbf{N}_{\mathbb{R}}}(\mathcal{C}_\pi)\to \mathscr{T}_{\mathbf{E}}(\mathcal{C}(A_{\mathbb{R},S})$.
In fact, there is a large commutative diagram of embeddings of cluster theories which relates many cluster theories of type $\mathbb{A}$, including all those mentioned in the present paper so far \cite{rock2021cluster}.

Another cluster theory of type $\mathbb{A}$ was introduced by Kulkarni, Matherne, Mousavand, and the second author \cite{kulkarni2021associahedron} by constructing a continuous associahedron of type $\mathbb{A}$.
We denote this cluster theory by $\mathscr{T}_{\mathbf{T}}(\mathcal{C}_\mathcal{Z})$.
By \cite[Theorem 6.18]{kulkarni2021associahedron} and \cite[Section 3.1]{rock2021cluster}, these are embeddings of cluster theories of type $\mathbb{A}_n$ into $\mathscr{T}_{\mathbf{T}}(\mathcal{C}_\mathcal{Z})$.
As of writing, it is not known if there is an embedding of cluster theories that relates $\mathscr{T}_{\mathbf{T}}(\mathcal{C}_\mathcal{Z})$ to either $\mathscr{T}_{\mathbf{N}_{\mathbb{R}}}(\mathcal{C}_\pi)$ or $\mathscr{T}_{\mathbf{E}}(\mathcal{C}(A_{\mathbb{R},S})$.

\section{Families of infinite type $\mathbb{D}$ (weak) cluster categories}\label{sec:main}
In this section, we introduce the main construction of the paper. We will first introduce our main combinatorial objects of punctured $\infty$-gons, which are natural generalizations of the combinatorial notions in \cite{schiffler2008geometric} for the finite case. We then show that our proposed combinatorial models are correct, in the sense that, in each case, there is a family of infinite type $\mathbb{D}$ cluster categories such that the cluster structure of each category is encoded in the corresponding punctured $\infty$-gon.
We note that our construction differs from Igusa and Todorov's in \cite{igusa2013continuous} as their construction produces a category that is either finite or continuous and our construction is neither.

\subsection{Punctured infinity-gons}\label{subsec:tagged_edges}

\begin{definition}\label{def:punctured infty-gon}
Let $D$ be the closed $2$-dimensional ball $D = \{(x,y) \in \RR^2 \text{ such that } ||(x,y)|| \leq 1\} $. 
\begin{itemize}
    \item The \emph{punctured $\infty$-gon} $\Pinfty$ is given by $D$ together with a set of marked points $\mathcal{M}$ such that $\mathcal{M} \cap \interior D = \{(0,0)\}$ and $\mathcal{M} \cap \partial D \simeq \ZZ$ has no one-sided accumulation points and one two-sided accumulation point $a$ such that $a \not\in \mathcal{M}$. See \Cref{fig:inftygon} (left).
    
    \item For $n \in \NN_{>0}$, the \emph{punctured $(n,\infty)$-gon} $\mP_{n,\infty}$ is given by $D$ together with a set of marked points $M$ such that $\mathcal{M} \cap \interior D = \{(0,0)\}$ and $\mathcal{M} \cap \partial D \simeq \ZZ$ has no one-sided accumulation points and $n$ two-sided accumulation points $\{a_1,\dots, a_n\}$ such that $a_i \not\in \mathcal{M}$ for every $i=1,\dots,n$.  See \Cref{fig:inftygon} (right).
    
    \item The \emph{punctured completed $(n,\infty)$-gon}, denoted by $\mP_{\overline{n,\infty}}$, is defined in an analogous way, but the two-sided accumulation points $a_1,\dots, a_n$ are marked points, that is $a_1,\dots,a_n \in \mathcal{M}$. See \Cref{fig:completed inftygon}.
\end{itemize}
\end{definition}

We always label the central marked point of the $\infty$-gon by $(0,0)$. For the labels of the other marked points in the boundary, see the  examples in Figures~\ref{fig:inftygon} and \ref{fig:completed inftygon}.
\begin{figure}[h] \centering
\input{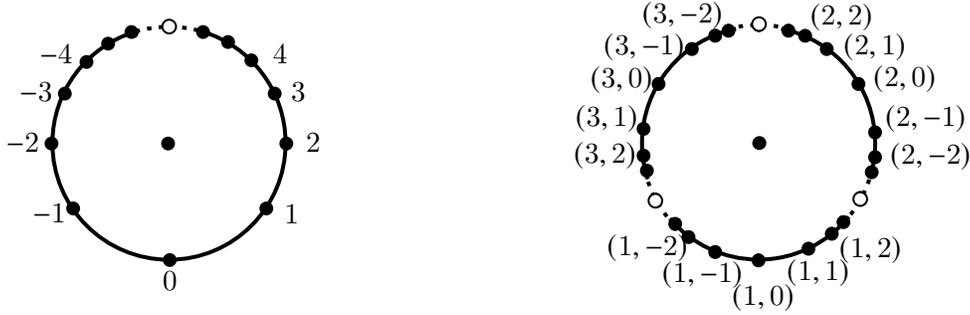}
\caption{On the left $\mP_\infty = \mP_{1,\infty}$ and on the right $\mP_{3,\infty}$ from Definition~\ref{def:punctured infty-gon}.}
\label{fig:inftygon}
\end{figure}

\begin{figure}[h]
    \centering
    \input{completed_infty_gons}
    \caption{On the left $\mP_{\overline{1,\infty}}$ and on the right $\mP_{\overline{3,\infty}}$. The blue marked points are two-sided accumulation points.}\label{fig:completed inftygon}
\end{figure}

\begin{definition}\label{def:tagged edges}
An \emph{edge} is a triple $((h,a),(k,b),\alpha)$ where $(h,a),(k,b) \in \mathcal{M}\backslash \{(0,0)\}$ are distinct marked points and $\alpha: [0,1] \to D$ is a path from $(h,a)$ to $(k,b)$ such that:
\begin{enumerate}[label= (\roman*)]
    \item $\alpha$ is homotopic to the counterclockwise path from $(h,a)$ to $(k,b)$ along $\partial D$;
    \item $\alpha((0,1)) \subset \interior(D)\backslash\{(0,0)\}$, that is the image of the path, except the starting and end points, is contained in the interior of the punctured $(n,\infty)$-gon;
    \item $\alpha$ does not cross itself, that is there are no $t_1,t_2 \in [0,1]$ such that $\alpha(t_1)=\alpha(t_2)$ unless $t_1=0, t_2=1$ or vice versa.
    \item $(k,b) \neq (h,a+1)$.
\end{enumerate}

Two edges $((h_1,a_1),(k_1,b_1),\alpha)$ and $((h_2,a_2),(k_2,b_2),\beta)$ are said to be \emph{equivalent} if $((h_1,a_1),(k_1,b_1)) = ((h_2,a_2),(k_2,b_2))$ and $\alpha$ is homotopic to $\beta$. Let $\mathcal E_{n,\infty}$ be the set of equivalence classes of edges in $\mP_{n,\infty}$ and note that an element in $\mathcal E_{n,\infty}$ is uniquely determined by a pair of marked points $((h,a),(k,b))$.
We write elements of $\mathcal E_{n,\infty}$ as $E_{(h,a),(k,b)}$.
Then we define the set of \emph{tagged edges} as:
\[ \mathcal E_{n,\infty}' = \{ {E_{(h,a),(k,b)}^\epsilon} \mid E_{(h,a),(k,b)} \in E, \epsilon= \pm 1 \text{ and } \epsilon =1 \text{ if } (h,a) \neq (k,b) \}. \]
For $\mP_{\overline{n,\infty}}$, we define the set of tagged edges similarly and denote it $\mathcal{E}'_{\overline{n,\infty}}$.
\end{definition}

\begin{notation}\label{note:puncture edges}
We will draw tagged edges with same start and endpoint $E_{(h,a),(h,a)}^\epsilon$ as segments from the marked point $(h,a)$ to the puncture in the middle of $\mP_{n,\infty}$, (or $\mP_{\overline{n,\infty}}$)  with a tag on it if $\epsilon=-1$. See Figure~\ref{fig:punctured infinity-gon triangulations}.
\end{notation}

Elementary moves and translation are operations on the set of tagged edges, which are defined below.
These are natural generalizations of the notions in \cite{schiffler2008geometric} from finite to the infinite case. Note that the vertices are always arranged in a cyclic anticlockwise order. 
For the completed $(n,\infty)$-gon, we use the convention $\infty \pm k = \infty$, for any $k \in \ZZ$.

\begin{definition}\label{def:elementary move}
An \textit{elementary move} is a pair of distinct tagged edges $\left(E_{(h_1,a_1),(k_1,b_1)}^\epsilon, E_{(h_2,a_2),(k_2,b_2)}^\lambda\right)$ such that one of the following conditions is satisfied:
\begin{enumerate}[label=(\roman*)]
    \item if $h_1=k_1$ and $b_1=a_1+2$, then 
    $E_{(h_2,a_2),(k_2,b_2)}^\lambda=E_{(h_1,a_1-1),(h_1,b_1)}$;
    \item if $h_1 \neq k_1$ or $\left( h_1=k_1 \text{ and }a < a+3\leq b < a-1 \text{ in the cyclic order}\right)$, then
    $E_{(h_2,a_2),(k_2,b_2)}^\lambda=E_{(h_1,a_1-1),(k_1,b_1)}$ or\\ $E_{(h_2,a_2),(k_2,b_2)}^\lambda =E_{(h_1,a_1),(k_1,b_1-1)}$;
    \item if $h_1=k_1$ and $b_1=a_1-1$, then
    $E_{(h_2,a_2),(k_2,b_2)}^\lambda= E_{(k_1,a_1),(k_1,b_1-1)}$ or $E_{(h_2,a_2),(k_2,b_2)}^\lambda= E_{(k_1,b_1),(k_1,b_1)}^{+1}$ or $E_{(h_2,a_2),(k_2,b_2)}^\lambda= E_{(k_1,b_1),(k_1,b_1)}^{-1}$;
    \item if $(h_1,a_1)=(k_1,b_1)$, then $E_{(h_2,a_2),(k_2,b_2)}^\lambda= E_{(h_1,a_1),(h_1,a_1-1)}$.
\end{enumerate}
\end{definition}

\begin{definition}\label{def:translation}
The \textit{translation} $\tau$ is a bijection on the set of tagged edges $\tau:\mathcal E_{n,\infty}\to \mathcal E_{n,\infty}$ (or $\mathcal E_{\overline{n,\infty}}\to \mathcal E_{\overline{n,\infty}}$) defined by
\[ \tau(E_{(h,a),(k,b)}^\epsilon) =\begin{cases} E_{(h,a+1),(k,b+1)} & \text{ if } (h,a) \neq (k,b)\\
E_{(h,a+1),(h,a+1)}^{-\epsilon} & \text{ if } (h,a)=(k,b). \end{cases}\]
\end{definition}

The following relation between elementary moves and translation holds as it does in the finite case.

\begin{lemma}\label{lem:elementary move and translation}
Let $E_{(h_1,a_1),(k_1,b_1)}^\epsilon, E_{(h_2,a_2),(k_2,b_2)}^\lambda$ be in $\mathcal E'_{n,\infty}$. Then we have that $(E_{(h_1,a_1),(k_1,b_1)}^\epsilon, E_{(h_2,a_2),(k_2,b_2)}^\lambda)$ is an elementary move if and only if $ (\tau E_{(h_2,a_2),(k_2,b_2)}^\lambda ,E_{(h_1,a_1),(k_1,b_1)}^\epsilon)$ is an elementary move.
\end{lemma}

In Figures~\ref{fig:translation quiver for infty-gon}, \ref{fig:translation quiver for (3,infty)-gon}, \ref{fig:translation quiver for completed infty-gon}, and \ref{fig:translation quiver for completed (3,infty)-gon} we see the quiver whose vertices are tagged edges and arrows are elementary moves for $\mP_\infty=\mP_{1,\infty}$, $\mP_{\overline{1,\infty}}$, $\mP_{3,\infty}$, and $\mP_{\overline{3,\infty}}$, respectively.
These are stable translation quivers.

\begin{figure}
\centering
    \input{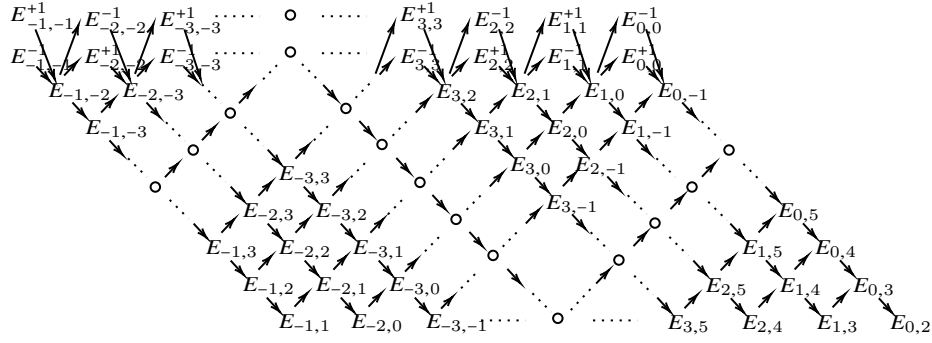}
    \caption{The stable translation quiver for $\mP_\infty=\mP_{1,\infty}$, where we have omitted the arrows from the last column to the first column.
    Each arrow correspond to an elementary move.
    The translation $\tau$ takes a vertex to the next vertex to the left without moving up or down.
    If $E_{i,j}$ has $i=-1$ and $j\neq -1$, then $\tau$ takes $E_{-1,j}$ to the vertex in the last column and same row: $E_{0,j+1}$.
    And, $\tau E^\epsilon_{-1,-1}=E^{-\epsilon}_{0,0}$.}
    \label{fig:translation quiver for infty-gon}
\end{figure}

\begin{figure}[h]
\centering
    \input{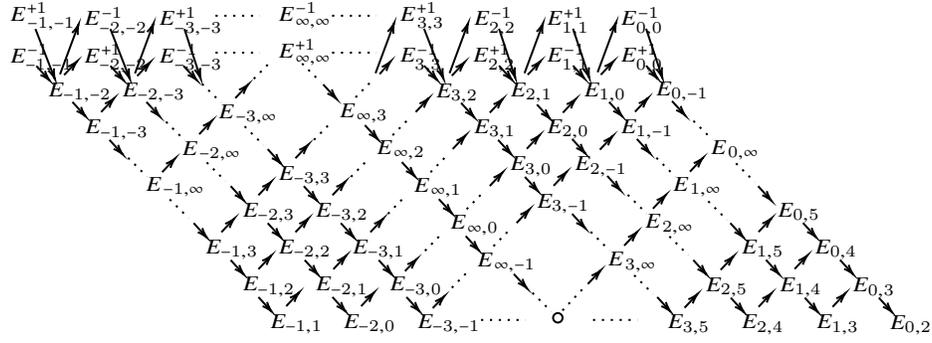}
    \caption{The stable translation quiver for $\mathcal P_{\overline{1,\infty}}$, where we have omitted the arrows from the last column to the first column.
    The translation $\tau$ slides $E^\epsilon_{i,j}$ up or down the diagonal if either $i$ or $j$ is $\infty$, but not both.
    The vertices $E^{-1}_{\infty,\infty}$ and $E^{+1}_{\infty,\infty}$ are swapped by $\tau$.
    Otherwise, $\tau$ acts the same as in Figure~\ref{fig:translation quiver for infty-gon}.}
    \label{fig:translation quiver for completed infty-gon}
\end{figure}

\begin{sidewaysfigure*}
    \includegraphics[width = 200mm]{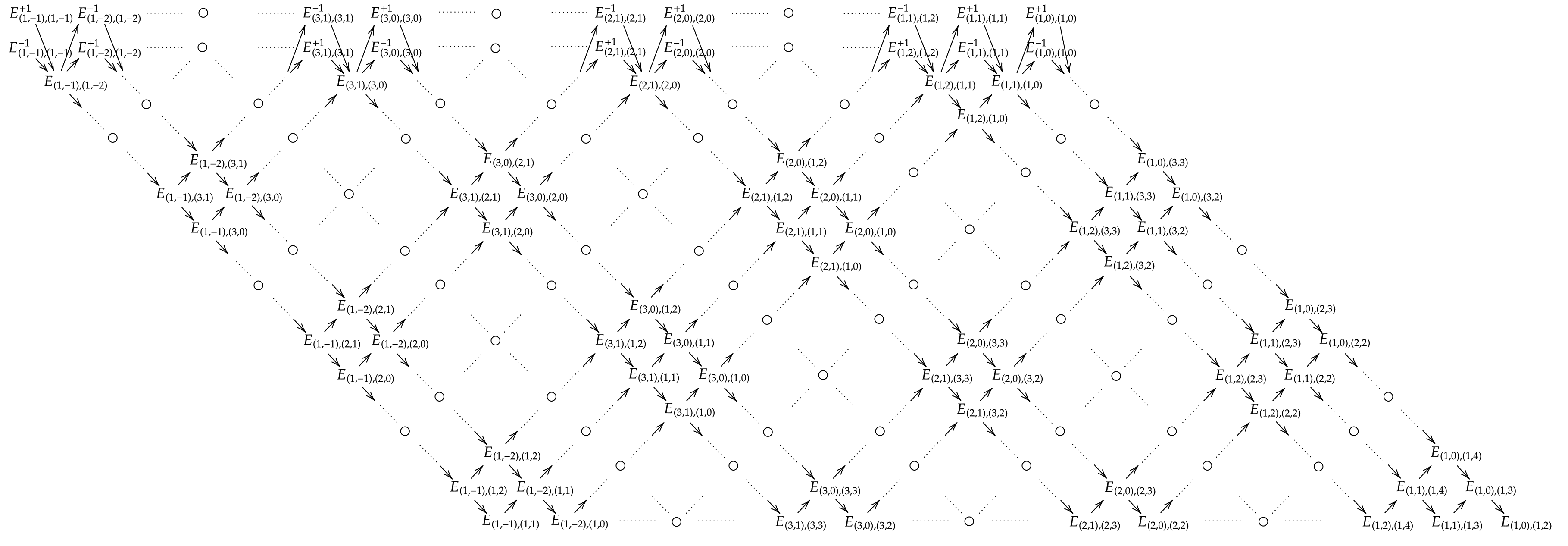}
    \caption{The stable translation quiver for $\mP_{3,\infty}$, where we have omitted the arrows from the last column to the first column.
    Each arrow correspond to an elementary move.
    The translation $\tau$ takes a vertex to the next vertex to the left without moving up or down.
    If $E_{(i,j)(k,l)}$ has $(i,j)=(1,-1)$ and $(k,l)\neq (1,-1)$, then $\tau$ takes $E_{(1,-1),(k,l)}$ to the vertex in the last column and same row: $E_{(1,0),(k,l+1)}$.
    And, $\tau E^\epsilon_{(1,-1)(1,-1)} = E^{-\epsilon}_{(1,0)(1,0)}$.}
    \label{fig:translation quiver for (3,infty)-gon}
\end{sidewaysfigure*}

\begin{sidewaysfigure*}
    \includegraphics[width = 200mm]{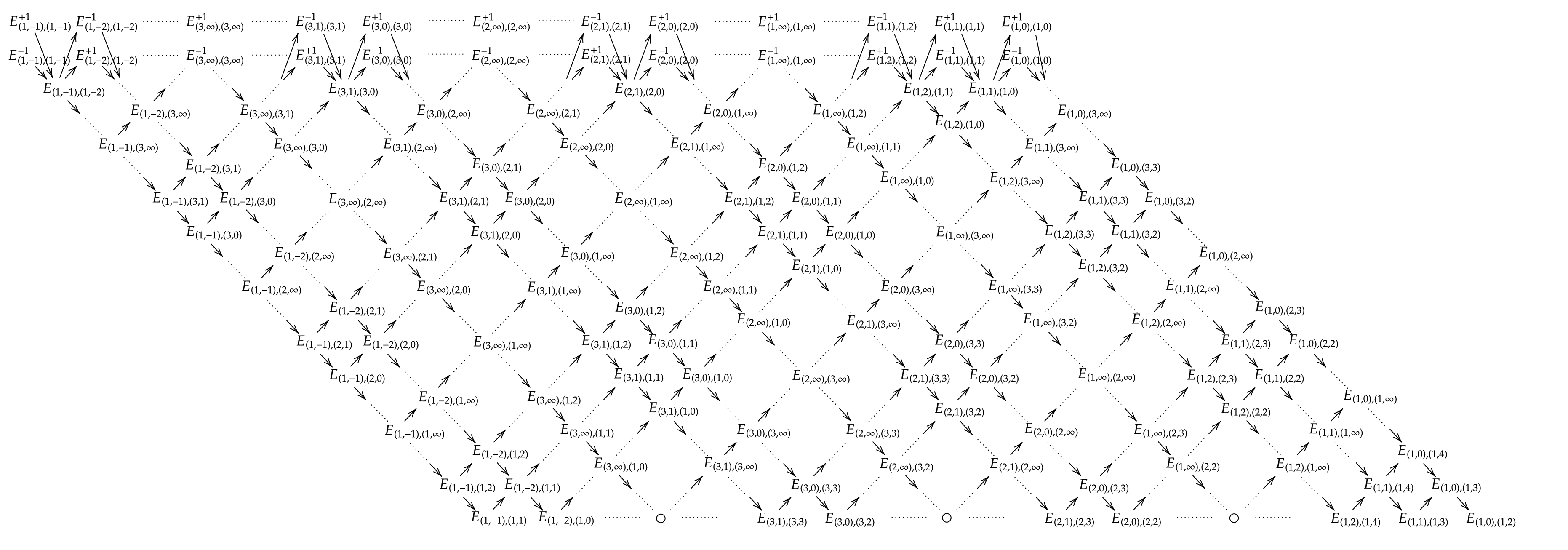}
    \caption{The stable translation quiver for $\mP_{\overline{3,\infty}}$, where we have omitted the arrows from the last column to the first column.
    Each arrow correspond to an elementary move.
    The translation $\tau$ slides vertices with $(i,\infty)$ in the subscript up and down the diagonals.
    Vertices of the form $E_{(i,\infty)(j,\infty)}$ are fixed by $\tau$ if $i\neq j$.
    Vertices of the form $E^\pm_{(i,\infty)(i,\infty)}$ are exchanged with $E^\mp_{(i,\infty)(i,\infty)}$.
    Otherwise, $\tau$ acts as it does in Figure~\ref{fig:translation quiver for (3,infty)-gon}.}
    \label{fig:translation quiver for completed (3,infty)-gon}
\end{sidewaysfigure*}

In order to introduce triangulations of $\infty$-gons, let us first recall the notion of crossing numbers. 

\begin{definition}\label{def:crossing number}
Let $E= E_{(h_1,a_1),(k_1,b_1)}^\epsilon$ and $F=E_{(h_2,a_2),(k_2,b_2)}^\lambda$  be two distinct tagged edges in $\mathcal E'_{n,\infty}$ (or $\mathcal E'_{\overline{n,\infty}}$). The \textit{crossing number} $\mathfrak{c}(E,F)$ is the number of intersections of $E$ and $F$ in the interior of the punctured $(n,\infty)$-gon $\interior(\mP_{n,\infty})$ (or completed $(n,\infty)$-gon $\interior(\mP_{n,\infty})$). More precisely:
\begin{enumerate}[label=(\roman*)]
    \item If $(h_1,a_1) \neq (k_1,b_1)$ and $(h_2,a_2) \neq (k_2,b_2)$, then
    \[ \mathfrak{c}(E,F) = \min\{ |\alpha \cap \beta \cap \interior(\mP_{n,\infty})| \mid ((h_1,a_1),(k_1,b_1),\alpha)\in E, ((h_2,a_2),(k_2,b_2),\beta)\in F\}; \]
    \item If $(h_1,a_1)=(k_1,b_1)$ and $(h_2,a_2)\neq (k_2,b_2)$, let $\alpha$ be the segment from $(h_1,a_1)$ to the puncture in $\mP_{n,\infty}$. Then
    \[ \mathfrak{c}(E,F) = \min\{ | \alpha \cap \beta \cap \interior(\mP_{n,\infty})| \mid ((h_2,a_2),(k_2,b_2),\beta)\in F \}; \]
    \item If $(h_1,a_1)\neq(k_1,b_1)$ and $(h_2,a_2)= (k_2,b_2)$, let $\beta$ be the segment from $(h_2,a_2)$ to the puncture in $\mP_{n,\infty}$. Then
    \[ \mathfrak{c}(E,F) = \min\{ | \alpha \cap \beta \cap \interior(\mP_{n,\infty})| \mid ((h_1,a_1),(k_1,b_1),\alpha)\in E \}; \]
    \item If $(h_1,a_1)=(k_1,b_1)$ and $(h_2,a_2)=(k_2,b_2)$, then
    \[ \mathfrak{c}(E,F) = \begin{cases} 1 & \text{ if } (h_1,a_1) \neq (h_2,a_2) \text{ and } \epsilon \neq \lambda\\
    0 & \text{ otherwise}. \end{cases}\]
\end{enumerate}
We say that $E$ \emph{crosses} $F$ if $\mathfrak{c}(E,F) \neq 0$.
\end{definition}

Note that, in Case (i), $\mathfrak{c}(E,F)\neq 0$ is equivalent to
\[(h_2,a_2)<(k_1,b_1)<(k_2,b_2)
\quad \text{and} \quad
(h_1,a_1)<(h_2,a_2)<(k_1,b_1),\]
in the cyclic order of the vertices on the boundary.

\begin{definition}\label{def:triangulation}
A \textit{triangulation} of $\mP_{n,\infty}$ (or $\mP_{\overline{n,\infty}}$) is a maximal set $T\subset \mathcal{E}'_{n,\infty}$ (or $T\subset\mathcal{E}'_{\overline{n,\infty}}$) of non-crossing tagged edges.
\end{definition} 

Note that any triangulation $T$ of $\mP_{n,\infty}$ contains infinitely many tagged edges. Indeed suppose by contradiction that $|T|$ is finite. Then $\supp(T)$ is finite, and so there exists $a \in \ZZ$ such that $(1,a),(1,a+1), (1,a+2) \not\in \supp(T)$. Since $\mathfrak{c}(E_{(1,a+2),(1,a)},E)=0$ for every $E \in T$, $T$ is not a maximal set of non-crossing tagged, i.e. $T$ is not a triangulation.

Note also that, unlike the finite case, it is possible to construct a triangulation of $\mP_{m,\infty}$ that does not contain any edges that are drawn touching the puncture.
See Figure~\ref{fig:punctured infinity-gon triangulations} (right) for an example.

\medskip

We now construct a family of triangulations of $\mP_{n,\infty}$. For any $h \in \{1,\dots,n\}$ and $a \in \ZZ$, let $T(a)$ be the set of tagged edges with starting point in $(h,a)$, that is
\[ T(h,a) = \{ E_{(h,a),(h,a)}^{+1}, E_{(h,a),(h,a)}^{-1} \} \cup \{ E_{(h,a),(k,b)} \mid (k,b) \in  \{1,\dots,n\} \times \ZZ, (k,b) \neq \{(h,a),(h,a-1)\} \}. \]

\begin{figure}[h]
    \centering
    \input{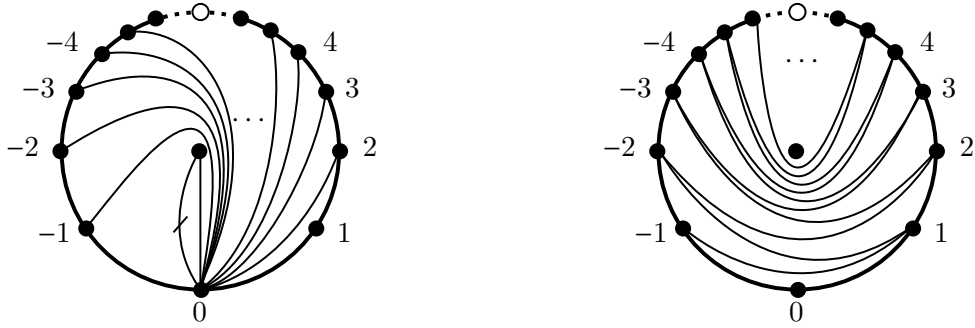}
    \caption{Two triangulations of $\Pinfty$.
    On the left is the triangulation $T(0)$.
    On the right is an example of a triangulation without an edge drawn to the puncture.}
    \label{fig:punctured infinity-gon triangulations}
\end{figure}

\begin{figure}[h]
    \centering
    \input{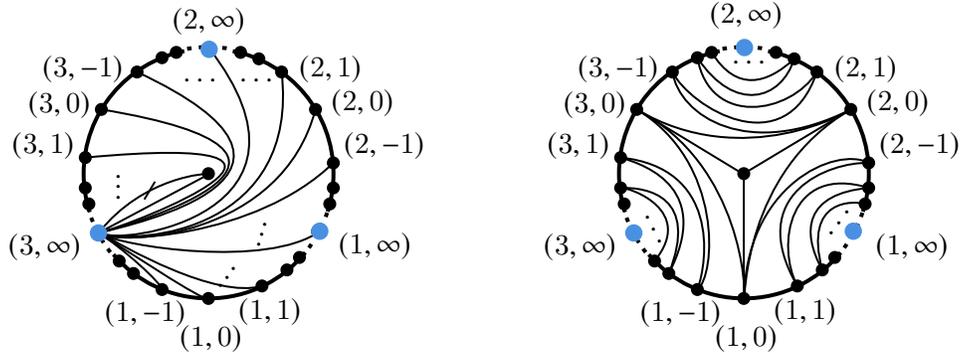}
    \caption{Two triangulations of $\mP_{\overline{3,\infty}}$.
    On the left is $T(3,\infty)$.
    On the right an example of a triangulation of without edges to the accumulation points on the boundary.
    }\label{fig:punctured 3-infinity-gon triangulations}
\end{figure}

\begin{lemma}
For any $h \in \{1,\dots,n\}$ and $a \in \ZZ$ ($a\in \ZZ\cup\{\infty\}$), $T(h,a)$ is a triangulation of $\mP_{n,\infty}$ (of $\mP_{\overline{n,\infty}}$).
\end{lemma}
\begin{proof}
We prove the lemma for $\mP_{n,\infty}$ as the proof for $\mP_{\overline{n,\infty}}$ is similar.

Suppose by contraction that there exists $E_{(k,c),(\ell,d)}^\lambda \in \mathcal{E}'_{n,\infty}\backslash T(h,a)$ such that $\mathfrak{c}(M,E_{c,d}^\lambda) = 0$ for every $M \in T$. Note that $(k,c) \neq (h,a)$ since $E_{(k,c),(\ell,d)}^\lambda \not\in T$. But then $E_{(h,a),(k,c-1)}\in T(h,a)$ since $(k,c-1) \neq a-1$ and $\mathfrak{c}(E_{(h,a),(k,c-1)}, E_{(k,c),(\ell,d)}) = 1$. It follows that $T(a)$ is a maximal set of non-crossing edges.
\end{proof}

\begin{remark}
Note that, there is no $m \in \ZZ$ such that $\tau^m = \id$.
More generally, there is no tagged edge $E \in \mathcal{E}_{n,\infty}$ (or $E \in \mathcal{E}_{\overline{n,\infty}}$) and there is no $m \in \ZZ$ such that $\tau^m(E)=E$.
That is, it is not possible to apply the translation $\tau$ a finite number of times to get back to the tagged edge we started with. This is one of the main differences with the finite case studied in \cite{schiffler2008geometric}.
\end{remark}

\begin{definition}\label{def:mutation on punctured infinity-gon}
    Let $T$ be a triangulation of $\mP_{n,\infty}$ (or $\mP_{\overline{n,\infty}}$).
    A \emph{mutation} is the act of removing some $E\in T$ and replacing it with $E^*\notin T$ such that $(T\setminus\{E\})\cup\{E^*\}$ is also a triangulation of $\mP_{n,\infty}$ (or $\mP_{\overline{n,\infty}}$).
\end{definition}

See Figures~\ref{fig:mutation on infty-gon}, \ref{fig:mutation on completed infty-gon}, and \ref{fig:mutation on completed (3,infty)-gon} for examples of mutation of triangulations in $\mP_\infty=\mP_{1,\infty}$, $\mP_{\overline{1,\infty}}$, and $\mP_{\overline{3,\infty}}$, respectively.

\begin{figure}[h]
    \centering
    \input{Mutation_triangulation_1,infty}
    \caption{Mutation of a triangulation of $\mP_{1,\infty}$.
    We replace the purple edge in the left picture with the red edge in the right picture.}
    \label{fig:mutation on infty-gon}
\end{figure}

\begin{figure}[h]
    \centering
    \input{Mutation_triangulation_compl_1,infty}
    \caption{Mutation of a triangulation of $\mP_{\overline{1,\infty}}$.
    We replace the purple edge in the left picture with the red edge in the right picture.}
    \label{fig:mutation on completed infty-gon}
\end{figure}

\begin{figure}[h]
    \centering
    \input{Mutation_triangulation_3,_infty}
    \caption{Mutation of a triangulation of $\mP_{3,\infty}$.
    We replace the purple edge in the left picture with the red edge in the right picture.}
    \label{fig:mutation on completed (3,infty)-gon}
\end{figure}

\subsection{Thread quivers of type $\mathbb{D}_{n,\infty}$ and $\mathbb{D}_{\overline{n,\infty}}$}\label{sec:D infinity}
\begin{definition}\label{def:D infinity}
    Let $i\in\NN_{>0}$ and let $Q$ be the quiver of type $D_4$ with one arrow in to the center point and two arrows out.
    Denote the arrow in by $\alpha$; denote the other two arrows by $\beta$ and $\gamma$.
We also denote $\mathbb{D}_{n,\infty}$ for the thread quiver obtained from $Q$ by threading $\alpha$ with $\mathcal{L}_i$.
    When $i=1$ we write $\mathbb{D}_\infty := \mathbb{D}_{1,\infty}$.
\end{definition}

We will label our quiver $\mathbb{D}_{n,\infty}$ as follows (with a bend to save space):\label{D}
\begin{displaymath}{\small 
  \xymatrix@C=2ex@R=3ex{
    (1,-1) \\
    (1,-2) \ar[u] \ar[d] &
    (1,-3) \ar[l] &
    \cdots \circ \cdots \ar[l] &
    (i,1) \ar[l] & (i,0) \ar[l] & (i,-1) \ar[l] &
    \cdots \ar[l] &
    (2,1) \ar[l] & (2,0) \ar[l] \\
    (1,-1'). & & & & & (1,1) \ar[r] & (1,2) \ar[r] & \cdots \circ \cdots \ar[r] & (2,-1) \ar[u]
    }}
\end{displaymath}
Using similar methods as in 
\cite{hanson2020decomposition}, we see that the indecomposable pointwise finite-dimensional representations of $\mathbb{D}_{n,\infty}$ have three forms.
\begin{enumerate}
    \item The first is $P_{(j,k)}$, a projective indecomposable. For $(1,1)$ and $(1,1')$ we have
    \begin{align*}
        P_{(1,-1)} &= \xymatrix@C=2ex@R=3ex{\Bbbk \\ 0 \ar[u] \ar[d] & 0 \ar[l] & \cdots \ar[l] & 0\ar[l] \\ 0}
        &
        P_{(1,-1')} &= \xymatrix@C=2ex@R=3ex{0 \\ 0 \ar[u] \ar[d] & 0 \ar[l] & \cdots \ar[l] & 0\ar[l] \\ \Bbbk}
    \end{align*}
    For all other $(j,k)$ we have
    \begin{displaymath}
        P_{(j,k)} = \xymatrix@C=2ex@R=3ex{\Bbbk \\ \Bbbk \ar[u] \ar[d] & \cdots \ar[l] & \Bbbk \ar[l] & 0 \ar[l] & \cdots \ar[l] & 0\ar[l] \\ \Bbbk,}
    \end{displaymath}
    where the furthest $\Bbbk$ to the right is at $(j,k)$.
    \item Next we have $M_{[j,k,\ell,m]}$ where there is a path $(\ell,m) \to (j,k)$ and $(\ell,m)$ is neither $(1,-1)$ nor $(1,-1')$.
    If $(\ell,m)=(1,1)$ then $M_{[j,k,\ell,m]}$ is injective.
    The $[j,k,\ell,m]$ subscript is short for the pair $(j,k)$ and $(\ell,m)$.
    When $(j,k)$ is either $(1,-1)$ or $(1,-1')$ we have
    \begin{align*}
        M_{[1,-1,\ell,m]} &= \xymatrix@C=2ex@R=3ex{\Bbbk \\ \Bbbk \ar[u] \ar[d] & \cdots \ar[l] & \Bbbk \ar[l] & 0 \ar[l] & \cdots \ar[l] & 0\ar[l] \\ 0}
        \\
        M_{[1,-1',\ell,m]} &= \xymatrix@C=2ex@R=3ex{0 \\ \Bbbk \ar[u] \ar[d] & \cdots \ar[l] & \Bbbk \ar[l] & 0 \ar[l] & \cdots \ar[l] & 0\ar[l] \\ \Bbbk,}
    \end{align*}
    where the last copy of $\Bbbk$ to the right is at $(\ell,m)$.
    For all other $(j,k)$, the representation $M_{[j,k,\ell,m]}$ is the bar representation from $(\ell,m)$ to $(j,k)$.
    \item Finally we have $M_{(j,k)^2(\ell,m)}$ where $(j,k)\notin\{(1,-1),(1,-1')\}$ is strictly to the left of $(\ell, m)$ in the following sense.
    \begin{displaymath}
        M_{(j,k)^2(\ell,m)} = \xymatrix@C=2ex@R=3ex{
        \Bbbk \\
        \Bbbk^2 \ar[u] \ar[d] & \cdots \ar[l] & \Bbbk^2 \ar[l] & \Bbbk \ar[l] & \cdots \ar[l] & \Bbbk \ar[l] & 0 \ar[l] & \cdots \ar[l] & 0 \ar[l] \\
        \Bbbk.
        }
    \end{displaymath}
   The furthest $\Bbbk^2$ to the right is at $(j,k)$ and the furthest $\Bbbk$ to the right is at $(\ell,m)$.
\end{enumerate}
These indecomposable representations are the infinite analogues to the indecomposable representations of a finite type $\mathbb{D}$ quiver.

\subsubsection{Local Dual}
In the classical case, one works with $\Hom_\Bbbk(M,\Bbbk)$ where $M$ is a finite-dimensional representation of a finite quiver.
However, this will not work for us.
In general, when we move a coproduct in the first variable of $\Hom$ to the outside, it turns into a product.
We do not want to work with $\prod_{n\in \mathbb{D}_{n,\infty}}\Hom_\Bbbk(M_n,\Bbbk)$.
Instead, we define our local dual to be $\bigoplus_{n\in \mathbb{D}_{n,\infty}}\Hom_\Bbbk(M_n,\Bbbk)$ on objects.
This definition yields a pair of local dual functors:
\begin{align*}
    D: \Rep_\Bbbk^\text{pwf}(\mathbb{D}_{n,\infty})&\to \Rep_\Bbbk^\text{pwf}(\mathbb{D}_{n,\infty}^\text{op})
    &
    D: \Rep_\Bbbk^\text{pwf}(\mathbb{D}_{n,\infty}^\text{op}) &\to
    \Rep_\Bbbk^\text{pwf}(\mathbb{D}_{n,\infty}).
\end{align*}
It is straight forward to check that both local dual functors are an equivalence of abelian categories.
This is a generalization of the local dual in the finite case where we pull out the direct sum before generalizing.

\subsubsection{Auslander--Bridger Transpose}
We know $\Hom$ behaves with respect to coproducts in the second variable.
Let $M$ be an object in $\rep_\Bbbk^\text{fp}(\mathbb{D}_{n,\infty})$.
Then $M$ has a projective resolution in $\rep_\Bbbk^\text{fp}(\mathbb{D}_{n,\infty})$, \[\bigoplus_{l=1}^n P_\ell \hookrightarrow \bigoplus_{j=1}^m P_j \twoheadrightarrow M,\] where each $P_j$ and $P_\ell$ are representable.
Then we may consider the cokernel $Tr M$ of the following map in $\rep_\Bbbk^\text{fp}(\mathbb{D}_{n,\infty})$: \[ TrM \twoheadleftarrow \bigoplus_{l=1}^n \left[\Hom_{\rep_\Bbbk^\text{fp}(\mathbb{D}_{n,\infty})}\left(P_\ell, \bigoplus_{k\in \mathbb{D}_{n,\infty}} P_k\right) \right]\hookleftarrow \bigoplus_{j=1}^m \left[\Hom_{\rep_\Bbbk^\text{fp}(\mathbb{D}_{n,\infty})}\left(P_j, \bigoplus_{k\in \mathbb{D}_{n,\infty}} P_k\right)\right].\]
It follows from similar proofs to the classical case that there are functors
\begin{align*}
    Tr: \rep_\Bbbk^\text{fp}(\mathbb{D}_{n,\infty}) &\to  \rep_\Bbbk^\text{fp}(\mathbb{D}_{n,\infty}^\text{op})
    &
    Tr: \rep_\Bbbk^\text{fp}(\mathbb{D}_{n,\infty}^\text{op}) &\to \rep_\Bbbk^\text{fp}(\mathbb{D}_{n,\infty}).
\end{align*}
These are \emph{not} equivalences.
In particular $Tr M = 0$ if and only if $M$ is projective.

\subsubsection{Auslander--Reiten translation}
Analogous to the finite case, we define $\tau:\rep_\Bbbk^\text{fp}(\mathbb{D}_{n,\infty})\to\rep_\Bbbk^\text{fp}(\mathbb{D}_{n,\infty})$ to be the composition $DTr$: \[ \rep_\Bbbk^\text{fp}(\mathbb{D}_{n,\infty}) \underbrace{\stackrel{Tr}{\to} \rep_\Bbbk^\text{fp}(\mathbb{D}_{n,\infty}^\text{op}) \stackrel{D}{\to}}_{\tau} \rep_\Bbbk^\text{fp}(\mathbb{D}_{n,\infty}).\]

By \cite[Proposition 7.20]{berg2014threaded}, $k\mathbb{D}_{n,\infty}$ is a semi-hereditary dualizing $k$-variety, for all $1\leq i < \infty$.
In particular, this means that $\mathcal{D}^b(\mathbb{D}_{n,\infty}):=\mathcal{D}^b(\rep_\Bbbk^\text{fp}(\mathbb{D}_{n,\infty}))$ is a triangulated Krull--Remak--Schmidt category whose indecomposable objects are shifts of indecomposable objects in $\rep_\Bbbk^\text{fp}(\mathbb{D}_{n,\infty})$.

Let $M[n]$ be an indecomposable object in $\mathcal{D}^b(\mathbb{D}_{n,\infty})$.
We define 
\[\tau M[n] := \begin{cases}
    (\tau M)[n] & M \text{ is not projective in }\rep_\Bbbk^\text{fp}(\mathbb{D}_{n,\infty}) \\
    I_j[n-1] & M=P_j \text{ is projective in }\rep_\Bbbk^\text{fp}(\mathbb{D}_{n,\infty}).
\end{cases}\]

\subsection{Cluster categories of infinite type $\mathbb{D}$}
From \cite{berg2014threaded} we know that there exists some Serre functor on $\mathcal{D}^b(\mathbb{D}_{n,\infty})$.
We will now show that this functor is $\tau^{-1}[1]$.

Let $M[m]$ and $N[n]$ be indecomposables in $\mathcal{D}^b(\mathbb{D}_{n,\infty})$.
Consider also $N'[n']:= \tau^{-1} N[n]$.
In $\rep_\Bbbk^\text{fp}(\mathbb{D}_{n,\infty})$ we have the projective resolutions
\begin{align*}
    P' \hookrightarrow &P \twoheadrightarrow M \\
    Q' \hookrightarrow &Q \twoheadrightarrow N \\
    R' \hookrightarrow &R \twoheadrightarrow N'
\end{align*}
where, $P$, $P'$, $Q$, $Q'$, $R$, and $R'$ are each a finite sum of representable indecomposable projectives.
In particular, representable indecomposable projectives have a top.
Let $k$ be the number of elements in the set \[ \left\{j\in \mathbb{D}_{n,\infty}\, \left|\, j\text{ satisfies any of }\begin{array}{l} j=t(\beta) \\ j=t(\gamma) \\ j=s(\beta)=s(\gamma) \\ j=s(\alpha) \\ j\text{ is the top of an ind.~sum'd~of }P\oplus P'\oplus Q\oplus Q'\oplus R\oplus R'\end{array} \right. \right\} \]
and consider the fully faithful, additive, and exact embedding $F:\rep_\Bbbk (\mathbb{D}_k)\to \rep_\Bbbk^\text{fp}(\mathbb{D}_{n,\infty})$ given by sending $P_i$ in $\rep_\Bbbk (\mathbb{D}_k)$ to $P_i$ in $\rep_\Bbbk^\text{fp}(\mathbb{D}_{n,\infty})$.
Lift this embedding to a fully faithful, additive, and triangulated embedding $\widetilde{F}:\mathcal{D}^b(\mathbb{D}_k)\to \mathcal{D}^b(\mathbb{D}_{n,\infty})$ given by $L[\ell]\mapsto (F L)[\ell]$.

We see that $M$, $N$, and $N'$ are in the image of $F$.
By overloading notation, let $M$, $N$, and $N'$ be indecomposables in $\rep_\Bbbk(\mathbb{D}_k)$ such that $FM\cong M$, $FN\cong N$, and $FN'\cong N'$.
(That is, we think of $\mathcal{D}^b(\mathbb{D}_k)$ as a thick subcategory of $\mathbb{D}^b(\mathbb{D}_{n,\infty})$.)
By definition, notice that
\begin{align*}
    \tau N[n] &\cong \tau F(N[n])\\
    &\cong F(\tau N[n]) \\
    &\cong F(N'[n']).
\end{align*}
We already know that, in the classical case,
\[ \Hom_{\mathcal{D}^b(\mathbb{D}_k)}(N[n],M[m])
\cong
\Hom_{\mathcal{D}^b(\mathbb{D}_k)}(M[m],\tau^{-1}N[n+1]).\]
Since $\widetilde{F}$ is fully faithful,
\[ \Hom_{\mathcal{D}^b(\mathbb{D}_{n,\infty})}(N[n],M[m])
\cong
\Hom_{\mathcal{D}^b(\mathbb{D}_{n,\infty})}(M[m],\tau^{-1}N[n+1]).\]
Therefore, $\tau^{-1}[1]$ is the desired Serre functor and $\mathcal{D}^b(\mathbb{D}_{n,\infty})$ is 2-Calabi--Yau.
The category $\mathcal{C}(\mathbb{D}_{n,\infty}):= \mathcal{D}^b(\mathbb{D}_{n,\infty}) / \tau^{-1}[1]$ is also triangulated.

\subsection{The combinatorial model is correct}

Recall the notion of $\mathcal{P}_{n,\infty}$ from \S\ref{subsec:tagged_edges}, which is the punctured disk with infinite marked points on its boundary and $i$ two-sided accumulation points of the marked points. 
Here we show $\tau$ on each type of indecomposable in $\mathcal{C}(\mathbb{D}_{n,\infty})$:
\begin{align*}
    \tau P_{(j,k)} &= P_{(j,k)}[1] & \tau P_{(j,k)}[1] &= I_{(j,k)}=M_{[j,k,1,1]} \\
    \tau M_{[j,k,\ell,m]} &=M_{[j,k+1,\ell,m+1]} & \tau M_{(j,k)^2(\ell,m)} &= M_{(j,k+1)^2(\ell,m+1)} &\text{if } (j,k)\neq (1,-2) \\
    \tau M_{[1,-2,\ell,m]} &= M_{(\ell,m+1)^2(1,1)} & \tau M_{(1,-2)^2(\ell,m)} &= P_{(\ell,m+1)} &\text{if }(\ell,m)\neq (1,-2)\\
    \tau M_{[1,-1,\ell,m]} &= M_{[1,-1',\ell,m+1]} & \tau M_{[1,-1',\ell,m]} &= M_{[1,-1,\ell,m+1]} &\text{if }(\ell,m)\neq(1,-2) \\
    \tau M_{[1,-1,1,-2]} &= P_{(1,-1')} &  \tau M_{[1,-1',1,-2]} &= P_{(1,-1)} \\
    \tau M_{[1,-2,1,-2]} &= P_{(1,1)}.
\end{align*}

The bijection $\Phi$ from $\ind(\mathcal{C}(\mathbb{D}_{n,\infty}))$ to tagged edges in $\mathcal{P}_{n,\infty}$ is given as follows.
First we consider indecomposables sent to tagged edges of the form $E_{**}^\pm$:
\begin{align*}
    P_{(1,-1)} &\mapsto E_{(1,-1)(1,-1)}^+ & P_{(1,-1')} &\mapsto E_{(1,-1)(1,-1)}^- \\
    M_{[1,-1,\ell,m]} &\mapsto E_{(\ell,m)(\ell,m)}^+ & M_{[1,-1',\ell,m]} &\mapsto E_{(\ell,m)(\ell,m)}^- \\
    P_{(1,-1)}[1] &\mapsto E_{(1,0)(1,0)}^- & P_{(1,-1')}[1] &\mapsto E_{(1,0)(1,0)}^+.
\end{align*}
Notice the difference in sign between the projectives and shifted projectives.

And now we cover the remaining tagged edges:
\begin{align*}
    P_{(j,k)} &\mapsto E_{(1,-1)(j,k)} &
    M_{[j,k,\ell,m]} &\mapsto E_{(\ell,m)(j,k+2)} \\
    M_{(j,k)^2(\ell,m)} &\mapsto E_{(j,k)(\ell,m)} &
    P_{(j,k)}[1] &\mapsto E_{(1,0)(j,k+1)}.
\end{align*}

\begin{proposition}\label{prop:phi}
    The defined function $\Phi$ is a bijection from $\ind(\mathcal{C}(\mathbb{D}_{n,\infty}))$ to tagged edges in $\mathcal{P}_{n,\infty}$.
    Furthermore, $\tau\Phi = \Phi\tau$ and elementary moves correspond to irreducible maps.
\end{proposition}
\begin{proof}
    The proposition follows from straightforward bookkeeping.
\end{proof}

\begin{proposition}\label{prop:ext equals crossing}
    Let $M, N$ be in $\ind(\mathcal{C}(\mathbb{D}_{n,\infty}))$.
    Then, $\Phi(M)$ crosses $\Phi(N)$ if and only if
    \[ \dim_{\Bbbk}(\Ext(M,N)\oplus \Ext(N,M)) > 0. \]
\end{proposition}
\begin{proof}
    We separate cases based on the types of the tagged edges $\Phi(M)$ and $\Phi(N)$.
    If neither $\Phi(M)$ nor $\Phi(N)$ go to the puncture, then the result follows from straightforward computations.
    This is similarly true if \emph{both} $\Phi(M)$ and $\Phi(N)$ go to the puncture.
    So, we prove the result when $\Phi(M)$ touches the puncture but $\Phi(N)$ does not.
    
    We classify $\Phi(M)$ into three cases and $\Phi(N)$ into four cases.
    \begin{align*}
        \text{(i)\ } & \Phi(M)=E_{(1,-1)(1,-1)}^\pm &
        \text{(a)\ } & \Phi(N)=E_{(1,-1)(j,k)} \\
        \text{(ii)\ } & \Phi(M)=E_{(1,0)(1,0)}^\pm &
        \text{(b)\ } & \Phi(N)=E_{(1,0)(j,k+1)} \\
        \text{(iii)\ } & \text{other} &
        \text{(c)\ } & \Phi(N) = E_{(l,m)(j,k+2)} \\
        & & \text{(d)\ } & \Phi(N) = E_{(j,k)(l,m)},
    \end{align*}
    where there is a path from $(l,m)$ to $(j,k)\notin\{(1,-1),(1,-1')\}$ in $\mathbb{D}_{n,\infty}$ for cases (c) and (d).
    \begin{itemize}
        \item[\textbf{(i)}]
    For case (i) we know $M=P_{(1,-1)}$ or $M=P_{(1,-1')}$.
    We will assume $M=P_{(1,-1)}$ as $M=P_{(1,-1')}$ is similar.
    
    \begin{itemize}
        \item[\textbf{(a)}]
    These tagged edges never cross.
    In this case $N=P_{(j,k)}$.
    In $\mathcal{C}(\mathbb{D}_{n,\infty})$, $P_{(1,-1)}$ and $P_{(j,k)}$ do not have any Ext groups.
    
   \item[\textbf{(b)}]
    These tagged edges do not cross since $(j,k)\neq (1,-1)$.
    In this case $N=P_{(j,k)}[1]$ where $(j,k)\notin\{(1,-1),(1,-1')\}$.
    In $\mathcal{C}(\mathbb{D}_{n,\infty})$, $P_{(1,-1)}$ and $P_{(j,k)}[1]$ do not have any Ext groups.
    
    \item[\textbf{(c)}]
    These tagged edges will only if and only if $(j,k)=(1,-2)$.
    In this case $N=M_{[j,k,l,m]}$.
    In $\mathcal{C}(\mathbb{D}_{n,\infty})$, $P_{(1,-1)}$ and $M_{[j,k,l,m]}$ have Ext groups if and only if $(j,k)=(1,-2)$.
    
    \item[\textbf{(d)}]
    These tagged edges always cross.
    Here $N=M_{(j,k)^2(l,m)}$.
    In $\mathcal{C}(\mathbb{D}_{n,\infty})$, the group $\Ext(M_{(j,k)^2(l,m)},P_{(1,-1)})\neq 0$.
    \end{itemize}
    
        \item[\textbf{(ii)}] Again we will assume $M=P_{(1,-1)}[1]$ as $M=P_{(1,-1')}[1]$ is similar.

   \begin{itemize}
       \item[\textbf{(a)}]
    These tagged edges always cross.
    In this case $N=P_{(j,k)}$.
    In $\mathcal{C}(\mathbb{D}_{n,\infty})$, the group $\Ext(P_{(1,-1)}[1],P_{(j,k)})\neq 0$.
    
    \item[\textbf{(b)}]
    These tagged edges do not cross.
    In this case $N=P_{(j,k)}[1]$ where $(j,k)\notin\{(1,-1),(1,-1')\}$.
    As with case (i)(a), in $\mathcal{C}(\mathbb{D}_{n,\infty})$ there are no Ext groups between $P_{(1,-1)}[1]$ and $P_{(j,k)}[1]$.
    
    \item[\textbf{(c)}]
    These tagged edges never cross.
    In this case $N=M_{[j,k,l,m]}$ where $(j,k)\notin\{(1,-1),(1,-1')\}$.
    In $\mathcal{C}(\mathbb{D}_{n,\infty})$ there are no Ext groups between $P_{(1,-1)}[1]$ and $M_{[j,k,l,m]}$.
    
    \item[\textbf{(d)}]
    These tagged edges always cross.
    In this case $N=M_{(j,k)^2(l,m)}$.
    Moreover, in $\mathcal{C}(\mathbb{D}_{n,\infty})$ the group $\Ext(M_{(j,k)^2(l,m)},P_{(1,-1)}[1])\neq 0$.
    \end{itemize} 
    
    \item[\textbf{(iii)}]
    We now assume $M=M_{[1,-1,j',k']}$ as $M=M_{[1,-1',j',k']}$ is similar, for some $(j',k')\notin\{(1,-1),(1,-1')\}$.
    
    \begin{itemize}
    \item[\textbf{(a)}]
    These tagged edges only cross if $(1,-1)< (j',k') <(j,k)$ in the cyclic order of the boundary of $\mathcal{P}_{n,\infty}$.
    In this case $N=P_{(j,k)}$.
    In $\mathcal{C}(\mathbb{D}_{n,\infty})$, we only have $\Ext(M,N)\neq 0$ if $(j',k')$ is to the left of $(j,k)$ in $\mathbb{D}_{i,\infty}$.
    \item[\textbf{(b)}]
    These tagged edges only cross if $(1,0)< (j',k') <(j,k+1)$ in the cyclic order of the boundary of $\mathcal{P}_{n,\infty}$.
    In this case $N=P_{(j,k)}[1]$ where $(j,k)\notin\{(1,-1),(1,-1')\}$.
    In $\mathcal{C}(\mathbb{D}_{n,\infty})$, we only have $\Ext(M,N)\neq 0$ or $\Ext(N,M)\neq 0$ if $(j',k')$ is equal to or to the left of $(j,k)$ in $\mathbb{D}_{i,\infty}$.
    \item[\textbf{(c)}]
    These tagged edges only cross if $(l,m) < (j',k') < (j,k+2)$ in the cyclic order of the boundary of $\mathcal{P}_{n,\infty}$.
    In this case $N=M_{[j,k,l,m]}$ where $(j,k)\notin\{(1,-1),(1,-1')\}$.
    Notice there is a path from $(l,m)$ to $(j',k')$ to $(j,k)$ in $\mathbb{D}_{n,\infty}$.
    In $\mathcal{C}(\mathbb{D}_{n,\infty})$, the group $\Ext(M_{[j,k,l,m]},M_{[1,-1,j',k']})\neq 0$ if and only if $(j',k')=(j,k+1)$ or $(j',k')$ is contained in the support of $M_{[j,k,l,m]}$ without $(l,m)$.
    The group $\Ext(M_{[1,-1,j',k']},M_{[j,k,l,m]})=0$.
    
    \item[\textbf{(d)}]
    These tagged edges only cross if $(j,k)<(j',k')<(l,m)$ in the cyclic order of the boundary of $\mathcal{P}_{n,\infty}$.
    There two possibilities for paths in $\mathbb{D}_{n,\infty}$.
    The first is $(l,m)$ to $(j,k)$ to $(j',k')$.
    The other is $(j',k')$ to $(l,m)$ to $(j,k)$.
    In this case $N=M_{(j,k)^2(l,m)}$.
    In $\mathcal{C}(\mathbb{D}_{n,\infty})$, the group $\Ext(M_{[1,-1,j',k']},M_{(j,k)^2(l,m)})\neq 0$ if and only if $(j',k')\neq (l,m)$ and there is a path $(j',k')$ to $(l,m)$ in $\mathbb{D}_{n,\infty}$.
    Moreover, we have that the group $\Ext(M_{(j,k)^2(l,m)},M_{[1,-1,j',k']})\neq 0$ if and only if $(j',k')\neq (j,k)$ and there is a path $(j,k)$ to $(j',k')$ in $\mathbb{D}_{n,\infty}$.
    \end{itemize}
    \end{itemize}
    
    \noindent \textbf{Conclusion.}
    We see that, in all cases, Ext-orthogonality is equivalent to the corresponding tagged edges not crossing, which completes the proof.
\end{proof}

We now state our first main result. 
\begin{theorem}\label{thm:main_result A body}
    There is a family of infinite type $\mathbb{D}$ cluster categories $\{\mathcal{C}(\mathbb{D}_{n,\infty}) \mid n\in\NN_{>0}\}$.
    Each cluster in each $\mathcal{C}(\mathbb{D}_{n,\infty})$ is weakly cluster tilting and each $\mathcal{C}(\mathbb{D}_{n,\infty})$ has a cluster theory.
    Furthermore, for each $\mathcal{C}(\mathbb{D}_{n,\infty})$, the combinatorial data is encoded in $\mathcal{P}_{n,\infty}$.
\end{theorem}

As we have not included the computations to show that mutation is given by left- and right-approximations, we present the following statement as a conjecture.

\begin{conjecture}\label{conj:main_result A}
    The weakly cluster tilting subcategories of each cluster category $\mathcal{C}(\mathbb{D}_{n,\infty})$ form a weak cluster structure.
\end{conjecture}

\subsection{Infinite type $\mathbb{D}$ weak cluster categories  associated to $\mathbb{D}_{\overline{n,\infty}}$}
\begin{definition}\label{def:D infinity closed}
Consider $\mathbb{D}_{n,\infty}$ as in Definition~\ref{def:D infinity}.
The quiver $\mathbb{D}_{\overline{n,\infty}}$ is obtained by adding a point between every copy of $\ZZ$ or $\NN$ in the thread of $\mathbb{D}_{n,\infty}$.
For vertices that are not accumulation points, we index the vertices of $\mathbb{D}_{\overline{n,\infty}}$ the same way as for $\mathbb{D}_{n,\infty}$.
The $(j,\infty)$ vertices are indexed in increasing order from right to left:
\begin{displaymath}
    \xymatrix@C=1.8ex@R=3ex{
    \bullet \\
    \bullet \ar[u] \ar[d] &
    \bullet \ar[l] &
    \cdots (i,\infty) \cdots \ar[l] &
    \bullet \ar[l] &  \bullet \ar[l] &
    \cdots \ar[l] \cdots  (2,\infty) \cdots &
    \bullet \ar[l] &  \bullet \ar[l] &
    \cdots (1,\infty) \cdots \ar[l] &
    \bullet \ar[l] & \bullet \ar[l] \\
    \bullet
    }
\end{displaymath}
\end{definition}

There are representable projectives $P_{(j,\infty)}$ at each $(j,\infty)$.
In $\Rep^{\text{pwf}}(\mathbb{D}_{\overline{n,\infty}})$, the radical of each $P_{(j,\infty)}$ is projective but not representable.
Thus, there are no indecomposable bar representations in $\rep^{\text{fp}}(\mathbb{D}_{\overline{n,\infty}})$ from $(j,\infty)$ to some $(\ell,m)$, for any $1\leq j \leq i$.
Instead, there are indecomposable bar representations whose support is the half open interval with infimum $(j,\infty)$ and maximum $(\ell,m)$ for some $(\ell,m)$.
Note that the infimum $(j,\infty)$ is \emph{not} in the support of these modules.
We denote such a module by $M_{(j,\infty,\ell,m]}$ and its support is shown below
\begin{displaymath}
        M_{(j,\infty,\ell,m]} =
        \xymatrix@C=2ex@R=3ex{
        0 \\
        0 \ar[u] \ar[d] &
        0 \ar[l] &
        \cdots 0 \cdots  \ar[l] &
        \Bbbk \ar[l]&
        \cdots \ar[l] &
        \Bbbk  \ar[l] &
        0 \ar[l] &
        \cdots \ar[l] &
        0 \ar[l] \\
        0.
        }
\end{displaymath}
The furthest $\Bbbk$ to the right is permitted to be at $(\ell,\infty)$ for some $1\leq \ell<j$.
Moreover, we notice that, in $\rep^{\text{fp}}(\mathbb{D}_{\overline{n,\infty}})$, there is no simple module at $(j,\infty)$ for any $1\leq j \leq i$.

The rest of the indecomposables described for $\rep^\text{fp}(\mathbb{D}_{n,\infty})$ are defined analogously using the accumulation vertices of $\mathbb{D}_{\overline{n,\infty}}$ in precisely the same way.

In order to obtain $\mathcal{C}(\mathbb{D}_{\overline{n,\infty}})$, we take inspiration from \cite{paquette2021completions}.
First construct $\mathcal{C}(\mathbb{D}_{2i,\infty})$.
Then, formally invert any morphism in $\mathcal{C}(\mathbb{D}_{2i,\infty})$ whose cone comes from objects in $\rep^{\text{fp}}(\mathbb{D}_{2i,\infty})$ with support entirely contained in $\{(2j,k)\mid k\in\ZZ\}$, for some $1\leq j\leq i$.
Extend to the rest of $\mathcal{C}(\mathbb{D}_{2i,\infty})$ bilinearly.
Denote this localised category by $\mathcal{C}(\mathbb{D}_{\overline{n,\infty}})$ and denote the quotient functor by $\pi$.

We now define a partial translation $\tau$ on $\mathcal{C}(\mathbb{D}_{\overline{n,\infty}})$.
There is an inclusion $\iota:\mathcal{C}(\mathbb{D}_{\overline{n,\infty}})\hookrightarrow \mathcal{C}(\mathbb{D}_{2i,\infty})$.
On indecomposables that come from projectives in $\rep^{\text{fp}}(\mathbb{D}_{\overline{n,\infty}})$:
\begin{align*}
    P_{(j,k)}[n] &\stackrel{\iota}{\mapsto} P_{(2(j-1)+1,k)}[n] & P_{(j,\infty)}[n] &\stackrel{\iota}{\mapsto} P_{(2j,0)}[n],
\end{align*}
where $1\leq j \leq i$, $k\in\ZZ$, and $n\in\{0,1\}$.
The rest of the inclusion is defined using the additive and triangulated structure of $\mathcal{C}(\mathbb{D}_{\overline{n,\infty}})$, since each indecomposable is the cone of a morphism of (shifts of) projectives.

Straightforward computations show that $\pi\iota$ is a triangulated equivalence on $\mathcal{C}(\mathbb{D}_{\overline{n,\infty}})$ and, in particular, $\pi\iota(M)\cong M$ for any object $M$ in $\mathcal{C}(\mathbb{D}_{\overline{n,\infty}})$.
Thus, we define \[\tau_{\mathcal{C}(\mathbb{D}_{\overline{n,\infty}})}:= \pi \tau_{\mathcal{C}(\mathbb{D}_{2i,\infty})} \iota.\]

Notice that for any $M_*$ we have defined, that does not use some $(j,\infty)$ in the $*$, the value of $\tau M_*$ is precisely as we previously saw.
For $M_*$ with a $(j,\infty)$, $(\ell,\infty)$, or both in the $*$ we have:
\begin{align*}
    \tau P_{(j,\infty)} &= P_{(j,\infty)}[1] &
    \tau P_{(j,\infty)}[1] &= M_{(j,\infty,1,1]} \\
    \tau M_{(j,\infty,\ell,m]} &= M_{(j,\infty,\ell,m+1]} &
    \tau M_{[j,k,\ell,\infty]} &= M_{[j,k+1,\ell,\infty]} & \text{if }(j,k)\notin\{(1,-1),(1,-1')\} \\
    \tau M_{(j,\infty)^2(\ell,m)} &= M_{(j,\infty)^2(\ell,m+1)} &
    \tau M_{(j,k)^2(\ell,\infty)} &= M_{(j,k+1)^2(\ell,\infty)} \\
    \tau M_{(j,\infty,\ell,\infty]} &= M_{(j,\infty,\ell,\infty]} &
    \tau M_{(j,\infty)^2(\ell,\infty)} &= M_{(j,\infty)^2(\ell,\infty)} \\
    \tau M_{[1,-1,\ell,\infty]} &= M_{[1,-1',\ell,\infty]} &
    \tau M_{[1,-1',\ell,\infty]} &= M_{[1,-1,\ell,\infty]}.
\end{align*}
Note that the indecomposable module $M_{(j,\infty,1,1]}$ is injective in $\rep^{\text{fp}}(\mathbb{D}_{\overline{n,\infty}})$ and so we retain the property of sending shifted projectives to injectives via $\tau$.

We now extend the bijection $\Phi$ from $\ind(\mathcal{C}(\mathbb{D}_{n,\infty}))$ to the tagged edges of $\mathcal{P}_{n,\infty}$.
The only changes in the definition of $\Phi$ (from the paragraph before Proposition~\ref{prop:phi}) are
\begin{align*}
    M_{(j,\infty,\ell,m]} &\mapsto E_{(\ell,m)(j,\infty)} & 
    P_{(j,\infty)}[1] &\mapsto E_{(1,0),(j,\infty)},
\end{align*} since there is no way to incorporate ``$+1$'' or ``$+2$'' with $\infty$.
In all other cases, treat $(j,\infty)$ or $(\ell,\infty)$ like any other vertex.
We note that the ``missing'' tagged edges for $\mathcal{P}_{\overline{n,\infty}}$ are precisely where each $(j,\infty)$ simple would be sent, for $1\leq j \leq i$.
See, for example, the empty spaces in the bottom row of Figures~\ref{fig:translation quiver for completed infty-gon} and \ref{fig:translation quiver for completed (3,infty)-gon} on pages \pageref{fig:translation quiver for completed infty-gon} and \pageref{fig:translation quiver for completed (3,infty)-gon}, respectively.

The following propositions also follow from straightforward bookkeeping.

\begin{proposition}
    The function $\Phi$ defined above is a bijection from $\ind(\mathcal{C}(\mathbb{D}_{\overline{n,\infty}}))$ to tagged edges in $\mathcal{P}_{\overline{n,\infty}}$.
    Furthermore, $\tau\Phi = \Phi\tau$ and elementary moves correspond to irreducible maps.
\end{proposition}

\begin{proposition}
    Let $M,N$ be in $\ind(\mathcal{C}(\mathbb{D}_{\overline{n,\infty}}))$.
    \begin{enumerate}
        \item[{\rm 1.}] Assume that either (i) one of $\Phi(M)$ or $\Phi(N)$ does not touch an accumulation point or (ii) $\Phi(M)$ and $\Phi(N)$ do not touch the same accumulation point.
        Then $\Phi(M)$ crosses $\Phi(N)$ if and only if
        \[\dim_{\Bbbk}(\Ext(M,N)\oplus\Ext(N,M)) > 0. \]
        \item[{\rm 2.}] If $\Phi(M)$ and $\Phi(N)$ touch the same accumulation point, then $\Phi(M)$ crosses $\Phi(N)$ if and only if
        \[\dim_{\Bbbk}(\Ext(M,N)\oplus\Ext(N,M)) > 1. \]
    \end{enumerate}
\end{proposition}

Notice that limiting edges in a triangulation may not be mutated.
These only occur when the accumulation points on the boundary are also marked, but not \emph{always}.

\medskip

We may now state our second main result.

\begin{theorem}\label{thm:main_result B body}
    There is a family of infinite type $\mathbb{D}$ weak cluster categories $\{\mathcal{C}(\mathbb{D}_{\overline{n,\infty}}) \mid n\in\NN_{>0}\}$.
    Each $\mathcal{C}(\mathbb{D}_{\overline{n,\infty}})$ has a cluster theory.
    Furthermore, for each $\mathcal{C}(\mathbb{D}_{\overline{n,\infty}})$, the combinatorial data is encoded in $\mathcal{P}_{\overline{n,\infty}}$.
\end{theorem}

We include a similar conjecture to Conjecture~\ref{conj:main_result A} by restricting to the weakly cluster tilting subcategories in each $\mathcal{C}(\mathbb{D}_{\overline{n,\infty}})$.

\begin{conjecture}\label{conj:main_result B}
    The weakly cluster tilting subcategories of each weak cluster category $\mathcal{C}(\mathbb{D}_{\overline{n,\infty}})$ form a weak cluster structure.
\end{conjecture}

\medskip

\noindent\textbf{Acknowledgements.} F.M. and F.Z. were partially supported by the grants G0F5921N (Odysseus programme) and G023721N from the Research Foundation - Flanders (FWO), and the UGent BOF grant STA/201909/038. J.D.R. was supported by UGent BOF grant 01P12621.

\small{
\bibliographystyle{abbrv}
\bibliography{Cluster.bib}
}

\bigskip
\noindent
  {\bf Authors' addresses:}

 \bigskip 

   \noindent Department of Computer Science, KU Leuven, Celestijnenlaan 200A, B-3001 Leuven, Belgium\\ 
 Department of Mathematics and Statistics,
 UiT – The Arctic University of Norway, 9037 Troms\o, Norway
 \\ E-mail address: {\tt fatemeh.mohammadi@kuleuven.be}

 \medskip

\noindent Department of Mathematics W16, Ghent University, Ghent, 9000, Belgium  \\ E-mail address: {\tt job.rock@ugent.be}

\medskip
\noindent Department of Mathematics, KU Leuven, Celestijnenlaan 200B, B-3001 Leuven, Belgium
 \\ E-mail address: {\tt 
frazaffa@gmail.com}

\end{document}